\documentclass[a4paper, 12pt]{amsart}
\usepackage{amsfonts, amsthm, amssymb, amsmath, stmaryrd}
\usepackage{mathrsfs,array}
\usepackage{xy}
\usepackage{hyperref}
\usepackage{verbatim}
\input xy
\xyoption{all}

\newenvironment{altenumerate}
   {\begin{list}
      {\textup{(\theenumi)} }
      {\usecounter{enumi}
       \setlength{\labelwidth}{0pt}
       \setlength{\labelsep}{2pt}
       \setlength{\leftmargin}{0pt}
       \setlength{\itemsep}{\the\smallskipamount}
       \renewcommand{\theenumi}{\roman{enumi}}
      }}
   {\end{list}}

\newtheorem{lem}{Lemma}[section]

\newtheorem{definition}[lem]{Definition}

\newtheorem{cor}[lem]{Corollary}
\newtheorem{thm}[lem]{Theorem}
\newtheorem{prop}[lem]{Proposition}

\newtheorem{conj}[lem]{Conjecture}
\newtheorem{claim}[lem]{Claim}

\theoremstyle{remark}
\newtheorem{rem}[lem]{Remark}
\newtheorem{example}[lem]{Example}

\DeclareMathOperator{\Hom}{Hom}

\DeclareMathOperator{\coker}{coker}

\DeclareMathOperator{\Spa}{Spa}
\DeclareMathOperator{\Spec}{Spec}
\DeclareMathOperator{\Spf}{Spf}
\DeclareMathOperator{\Fil}{Fil}
\DeclareMathOperator{\gr}{gr}
\DeclareMathOperator{\Gal}{Gal}

\DeclareMathOperator{\Lie}{Lie}
\DeclareMathOperator{\qlog}{qlog}

\def\ad{\mathrm{ad}}
\def\an{\mathrm{an}}
\def\et{\mathrm{\acute{e}t}}

\def\proet{\mathrm{pro\acute{e}t}}

\def\cont{\mathrm{cont}}

\def\cris{\mathrm{cris}}
\def\dR{\mathrm{dR}}

\newcommand{\Z}{\mathbb{Z}}
\newcommand{\F}{\mathbb{F}}
\newcommand{\Q}{\mathbb{Q}}

\newcommand{\B}{\mathbb{B}}
\newcommand{\C}{\mathbb{C}}
\newcommand{\OO}{\mathcal{O}}

\newcommand{\GL}{\mathrm{GL}}
\newcommand{\PGL}{\mathrm{PGL}}

\newcommand{\Sh}{\mathrm{Sh}}

\newdir{ >}{{}*!/-5pt/@{>}}
\SelectTips{cm}{10}

\begin{document}

\author{Peter Scholze}
\title{Perfectoid Spaces: A survey}
\address{Mathematisches Institut der Universit\"at Bonn, Endenicher Allee 60, 53115 Bonn, Germany}
\email{scholze@math.uni-bonn.de}

\begin{abstract} This paper, written in relation to the Current Developments in Mathematics 2012 Conference, discusses the recent papers on perfectoid spaces. Apart from giving an introduction to their content, it includes some open questions, as well as complements to the results of the previous papers.
\end{abstract}\footnote{This work was done while the author was a Clay Research Fellow.}

\maketitle
\tableofcontents
\pagebreak

\section{Introduction}

The original aim of the theory of perfectoid spaces was to prove Deligne's weight-monodromy conjecture over $p$-adic fields by reduction to the case of local fields of equal characteristic $p$, where the result is known. In order to so, the theory of perfectoid spaces establishes a general framework relating geometric questions over local fields of mixed characteristic with geometric questions over local fields of equal characteristic. One application of this theory is a general form of Faltings's almost purity theorem. For the moment, the weight-monodromy conjecture is not proved in full generality using these methods; however, we can reduce the conjecture to a statement on approximating a 'fractal' by algebraic varieties.

However, the theory of perfectoid spaces has proved to be useful in other situations, and certainly many more will be found. On one hand, perfectoid spaces embody Faltings's almost purity theorem; as this is the crucial technical ingredient to Faltings's approach to $p$-adic Hodge theory, it is not surprising that one can prove new results in $p$-adic Hodge theory. For example, it becomes possible to analyze general (proper smooth) rigid-analytic varieties, instead of just algebraic varieties. That $p$-adic Hodge theory for rigid-analytic varieties should be possible was already conjectured by Tate, \cite{TatePDivGroups}, when he first conjectured the existence of a Hodge-Tate decomposition, and of course it aligns well with the situation over $\C$.

In these first papers, the perfectoid spaces had more of an auxiliary role. However, it turns out that many natural constructions, which so far could not be given any geometric meaning, are perfectoid spaces in a natural way: For example, Shimura varieties with infinite level at $p$, and Rapoport-Zink spaces (local analogues of Shimura varieties) with infinite level. For Rapoport-Zink spaces, there has been a duality conjecture (proved by Faltings) which says that certain 'dual' pairs of Rapoport-Zink spaces are isomorphic at infinite level. Until now, the formulation of such an isomorphism has been very ad hoc; however, it can now be formulated as an isomorphism of perfectoid spaces. This has all expected consequences, such as comparisons of \'etale cohomology.

In the case of Shimura varieties, interesting applications (related to torsion in the cohomology of locally symmetric varieties and Emerton's completed cohomology groups) arise, which are work in progress (\cite{ScholzeTorsion}); some of this is sketched at the end of this survey.

This paper was written in relation to the talks of the author at the Current Developments in Mathematics conference 2012 at Harvard. The author wants to thank the organizers heartily for the invitation, and the opportunity to write this survey. The original intention was to write a survey paper about perfectoid spaces and the weight-monodromy conjecture. However, as this is exactly the content of \cite{ScholzePerfectoidSpaces1}, the author decided instead to give some introduction to the general content of the three papers \cite{ScholzePerfectoidSpaces1}, \cite{ScholzeHodge} and \cite{ScholzeWeinstein}, and mention some open questions, as well as some complements on the results of these papers. These are new results, but they are immediate applications of the theory built up there. Thus, some parts of this paper do not have the nature of a survey, and assume familiarity with the content of these papers.
\newpage

\section{Perfectoid Spaces}

\subsection{Introduction}

This introduction is essentially identical to a post of the author on MathOverflow, \cite{ScholzeMOPost}.

\begin{definition} A perfectoid field $K$ is a complete non-archimedean field $K$ of residue characteristic $p$, equipped with a non-discrete valuation of rank $1$, such that the Frobenius map $\Phi: \OO_K/p\to \OO_K/p$ is surjective, where $\OO_K\subset K$ is the subring of elements of norm $\leq 1$.
\end{definition}

Some authors, e.g. Gabber-Ramero in their book on almost ring theory, \cite{GabberRamero}, call such fields deeply ramified (although they do not require that they are complete).

\begin{example} Standard examples of perfectoid fields are given by the completion of $\mathbb{Q}_p(p^{1/p^\infty})$, $\mathbb{Q}_p(\mu_{p^\infty})$, $\overline{\mathbb{Q}}_p$, or $\mathbb{F}((t))(t^{1/p^\infty})$.
\end{example}

Given a perfectoid field $K$, one can form a second perfectoid field $K^\flat$, always of characteristic $p$, given as the fraction field of
\[
\OO_{K^\flat} = \varprojlim_\Phi \OO_K/p\ ,
\]
where the transition maps are given by Frobenius. Concretely, if $K$ is the completion of $\mathbb{Q}_p(p^{1/p^\infty})$, then $K^\flat$ is given by the completion of $\mathbb{F} _p((t))(t^{1/p^\infty})$, where $t$ is the element
\[
(p,p^{1/p},p^{1/p^2},\ldots)\in \OO_{K^\flat} = \varprojlim \OO_K/p\ .
\]
In particular, we have a canonical identification
\[
\OO_{K^\flat}/t = \F_p[t^{1/p^\infty}]/t\cong \Z_p[p^{1/p^\infty}]/p = \OO_K/p\ .
\]
   
In this situation, one has the following theorem, due to Fontaine-Wintenberger, \cite{FontaineWintenberger}, in most examples.

\begin{thm}\label{GalEquiv} There is a canonical isomorphism of absolute Galois group $\Gal(\bar{K}/K)\cong \Gal(\bar{K}^\flat/K^\flat)$.
\end{thm}

At this point, it may be instructive to explain this theorem in the example where $K$ is the completion of $\mathbb{Q}_p(p^{1/p^\infty})$; in all examples to follow, we make this choice of $K$. It says that there is a natural equivalence of categories between the category of finite extensions $L$ of $K$ and the category of finite extensions $M$ of $K^\flat$. Let us give an example: Say $M$ is the extension of $K^\flat$ given by adjoining a root of $X^2 - 7t X + t^5$. Basically, the idea is that one replaces $t$ by $p$, so that one would like to define $L$ as the field given by adjoining a root of $X^2 - 7p X + p^5$. However, this is obviously not well-defined: If $p=3$, then $X^2 - 7t X + t^5=X^2 - t X + t^5$, but $X^2 - 7p X + p^5\neq X^2 - p X + p^5$, and one will not expect in general that the fields given by adjoining roots of these different polynomials are the same.

However, there is the following way out: $M$ can be defined as the splitting field of $X^2 - 7t^{1/p^n} X + t^{5/p^n}$ for all $n\geq 0$ (using that $K^\flat$ is perfect), and if we choose $n$ very large, then one can see that the fields $L_n$ given as the splitting field of $X^2 - 7p^{1/p^n} X + p^{5/p^n}$ will stabilize as $n\rightarrow \infty$; this is the desired field $L$. Basically, the point is that the discriminant of the polynomials considered becomes very small, and the difference between any two different choices one might make when replacing $t$ by $p$ becomes comparably small.

This argument can be made precise by using Faltings's almost mathematics, as developed systematically by Gabber-Ramero, \cite{GabberRamero}. Consider $K\supset \OO_K\supset \mathfrak{m}$, where $\mathfrak{m}$ is the maximal ideal; in the example, it is the one generated by all $p^{1/p^n}$, and it satisfies $\mathfrak{m}^2 = \mathfrak{m}$, because the valuation on $K$ is non-discrete. We have a sequence of localization functors:
\[
\OO_K-\mathrm{mod}\to \OO_K-\mathrm{mod} / \mathfrak{m}-\mathrm{torsion}\to \OO_K-\mathrm{mod} / p-\mathrm{power\ torsion}\ .
\]
The last category is equivalent to $K$-mod, and the composition of the two functors is like taking the generic fibre of an object with an integral structure.

In this sense, the category in the middle can be seen as a slightly generic fibre, sitting strictly between an integral structure and an object over the generic fibre. Moreover, an object like $\OO_K/p$ is nonzero in this middle category, so one can talk about torsion objects, neglecting only very small objects. The official name for this middle category is $\OO_K^a$-mod: almost $\OO_K$-modules.

This category is an abelian tensor category, and hence one can define in the usual way the notion of an $\OO_K^a$-algebra (= almost $\OO_K$-algebra), etc. . With some work, one also has notions of almost finitely presented modules and (almost) \'etale maps. In the following, we will often use the notion of an almost finitely presented \'etale map, which is the almost analogue of a finite \'etale map: We will use the term almost finite \'etale map in the following.

\begin{thm}[Tate(\cite{TatePDivGroups}), Gabber-Ramero(\cite{GabberRamero})] If $L/K$ is a finite extension, then $\OO_L/\OO_K$ is almost finite \'etale. Similarly, if $M/K^\flat$ is finite, then $\OO_M/\OO_{K^\flat}$ is almost finite \'etale.
\end{thm}

As an example, assume $p\neq 2$ and $L=K(p^{1/2})$. For convenience, we look at the situation at a finite level, so let $K_n=\mathbb{Q}_p(p^{1/p^n})$ and $L_n=K_n(p^{1/2})$. Then $\OO_{L_n} = \OO_{K_n}[X] / (X^2 - p^{1/p^n})$. To check whether this is \'etale, look at $f(X)= X^2 - p^{1/p^n}$ and look at the ideal generated by $f$ and its derivative $f^\prime$. This contains $p^{1/p^n}$, so in some sense $\OO_{L_n}$ is etale over $\OO_{K_n}$ up to $p^{1/p^n}$-torsion. Now take the limit as $n\rightarrow \infty$ to see that $\OO_L$ is almost etale over $\OO_K$.

In fact, in the case of equal characteristic, i.e. $K^\flat$, the theorem is easy. Indeed, there will be some big $N$ such that $\Omega^1_{\OO_M/\OO_{K^\flat}}$ is killed by $t^N$ (where $t\in K^\flat$ is some element with $|t|<1$). But $K^\flat$ is perfect, and thus also $M$; it follows that the whole situation is invariant under Frobenius. Thus, $\Omega^1_{\OO_M/\OO_{K^\flat}}$ is killed by $t^{N/p}$, and thus by $t^{N/p^2}$, ..., i.e. $t^{N/p^k}$ for all $k$. This is exactly saying that $\Omega^1_{\OO_M/\OO_{K^\flat}}$ is almost zero. From here, one can deduce that $\OO_M$ is almost finite \'etale over $\OO_{K^\flat}$.

Now we can prove Theorem \ref{GalEquiv} above:
\[\begin{aligned}
\{\mathrm{finite\ \acute{e}tale\ covers\ of\ }K\}&\cong \{\mathrm{almost\ finite\ \acute{e}tale\ covers\ of\ }\OO_K\}\\
&\cong \{\mathrm{almost\ finite\ \acute{e}tale\ covers\ of\ }\OO_K/p\}\\
&= \{\mathrm{almost\ finite\ \acute{e}tale\ covers\ of\ }\OO_{K^\flat}/t\}\\
&\cong \{\mathrm{almost\ finite\ \acute{e}tale\ covers\ of\ }\OO_{K^\flat}\}\\
&\cong \{\mathrm{finite\ \acute{e}tale\ covers\ of\ }K^\flat\}
\end{aligned}\]
Here, we use that almost finite \'etale covers lift uniquely over nilpotents. Also, one can always find some element $t\in K^\flat$ as in the example such that $\OO_K/p = \OO_{K^\flat} / t$.

Now we want to generalize the theory to the relative situation. Here, the basic claim is the following.

\begin{claim} The affine line $\mathbb{A}^1_{K^\flat}$ 'equals' $\varprojlim \mathbb{A}^1_K$, where the transition maps are the $p$-th power map.
\end{claim}

As a first step towards understanding this, we check this on points. Here it says that $K^\flat = \varprojlim K$. In particular, there should be map $K^\flat\rightarrow K$ by projection to the last coordinate, which is usually denoted $x\mapsto x^\sharp$ and again this can be explained in an example:

Say $x = t^{-1} + 5 + t^3$. Basically, we want to replace $t$ by $p$, but this is not well-defined. But we have just learned that this problem becomes less serious as we take $p$-power roots. So we look at $t^{-1/p^n} + 5 + t^{3/p^n}$, replace $t$ by $p$, get $p^{-1/p^n} + 5 + p^{3/p^n}$, and then we take the $p^n$-th power again, so that the expression has the chance of being independent of $n$. Now, it is in fact not difficult to see that
\[
\lim_{n\rightarrow \infty} (p^{-1/p^n} + 5 + p^{3/p^n})^{p^n}\in K
\]
exists, and this defines $x^\sharp\in K$. Now the map $K^\flat\rightarrow \varprojlim K$ is given by
\[
x\mapsto (x^\sharp,(x^{1/p})^\sharp,(x^{1/p^2})^\sharp,\ldots)\ .
\]
In order to prove that this is a bijection, just note that
\[
\OO_{K^\flat} = \varprojlim \OO_{K^\flat}/t^{p^n} = \varprojlim_\Phi \OO_{K^\flat}/t = \varprojlim_\Phi \OO_K/p \leftarrow \varprojlim_{x\mapsto x^p} \OO_K\ .
\]
Here, the last map is the obvious projection, and in fact is a bijection, which amounts to the same verification as that the limit above exists. Afterwards, just invert $t$ to get the desired identification.

One sees immediately that the explicit description of the isomorphism involves $p$-adic limits, so a formalization will necessarily need to use some framework of rigid geometry. In the paper \cite{KedlayaLiu} of Kedlaya-Liu, where they are doing closely related things, they choose to work with Berkovich spaces. The author favors the language of Huber's adic spaces, as this language is capable of expressing more (e.g., Berkovich only considers rank-$1$-valuations, whereas Huber considers also the valuations of higher rank). In the language of adic spaces, the spaces are actually locally ringed topological spaces (equipped with valuations) (and affinoids are open, in contrast to Berkovich's theory, making it easier to glue). There is an analytification functor $X\mapsto X^\ad$ from schemes of finite type over $K$ to adic spaces over $K$ (similar to the functor associating to a scheme of finite type over $C$ a complex-analytic space). Then we have the following theorem:

\begin{thm} There is a homeomorphism of underlying topological spaces $|(\mathbb{A}^1_{K^\flat})^\ad|\cong \varprojlim |(\mathbb{A}^1_K)^\ad|$.
\end{thm}

At this point, the following question naturally arises: Both sides of this homeomorphism are locally ringed topological spaces: So is it possible to compare the structure sheaves? There is the obvious problem that on the left-hand side, we have characteristic $p$-rings, whereas on the right-hand side, we have characteristic $0$-rings. How can one pass from one to the other side?

\begin{definition} A perfectoid $K$-algebra is a complete Banach $K$-algebra $R$ such that the set of power-bounded elements $R^\circ\subset R$ is bounded and the Frobenius map $\Phi: R^\circ/p\to R^\circ/p$ is surjective.
\end{definition}

Similarly, one defines perfectoid $K^\flat$-algebras $S$. The last condition is then equivalent to requiring that $S$ perfect, whence the name. Examples are $K$, any finite extension $L$ of $K$, and $K\langle T^{1/p^\infty}\rangle$, which is $\OO_K\langle T^{1/p^\infty}\rangle\otimes_{\OO_K} K$, where $\OO_K\langle T^{1/p^\infty}\rangle$ is the $p$-adic completion of $\OO_K[T^{1/p^\infty}]$.

Recall that in classical rigid geometry, one considers rings like $K\langle T\rangle$, which is interpreted as the ring of convergent power series on the closed annulus $|T|\leq 1$. Now in the example of the $\mathbb{A}^1$ above, we take $p$-power roots of the coordinate, so after completion the rings on the inverse limit are in fact perfectoid.

In characteristic $p$, one can pass from usual affinoid algebras to perfectoid algebras by taking the completed perfection; the difference between the two is small, at least as regards topological information on associated spaces: Frobenius is a homeomorphism on topological spaces, and even on \'etale topoi. This is also the reason that we did not take the inverse limit $\varprojlim \mathbb{A}^1_{K^\flat}$ above: It does not change the topological space. In order to compare structure sheaves, one should however take this inverse limit.

Now we can state the tilting equivalence.

\begin{thm} The category of perfectoid $K$-algebras is canonically equivalent to the category of perfectoid $K^\flat$-algebras.
\end{thm}

The functor is given by $R\mapsto R^\flat = (\varprojlim R^\circ/p)[t^{-1}]$. Again, one also has $R^\flat = \varprojlim R$ as multiplicative monoids, where the transition maps are the $p$-th power map, giving also the map $R^\flat\rightarrow R$, $f\mapsto f^\sharp$.

There are two different proofs for this. One is to write down the inverse functor, given by $S\mapsto W(S^\circ)\otimes_{W(\OO_{K^\flat})} K$, using the map
\[
\Theta: W(\OO_{K^\flat})\rightarrow K
\]
known from $p$-adic Hodge theory. The other proof is similar to what we did above for finite \'etale $K$-algebras: Perfectoid $K$-algebras are equivalent to almost $\OO_K$-algebras $A$ s.t. $A$ is flat, $p$-adically complete and Frobenius induces an isomorphism $A/p^{1/p}\cong A/p$; these are in turn equivalent to almost $\OO_K/p$-algebras $\overline{A}$ s.t. $\overline{A}$ is flat and Frobenius induces an isomorphism $\overline{A}/p^{1/p}\cong \overline{A}$. From here, one can go to $\OO_{K^\flat}/t$, and reverse the steps to get to $K^\flat$.

The first identification between perfectoid $K$-algebras and certain almost $\OO_K$-algebras is not difficult. The second identification between certain almost $\OO_K$-algebras and certain almost $\OO_K/p$-algebras relies on the fact (already in the book by Gabber-Ramero) that the cotangent complex $\mathbb{L}_{\overline{A}/(\OO_K/p)}$ vanishes, and hence one gets unique deformations of objects and morphisms. At least on differentials $\Omega^1$, this is easy to see: Every element $x$ has the form $y^p$ because Frobenius is surjective; but then $dx = dy^p = py^{p-1}dy = 0$ because $p=0$ in $\overline{A}$.

Finally, we summarize briefly the main theorems on the basic nature of perfectoid spaces. An affinoid perfectoid space is associated to a perfectoid affinoid $K$-algebra, which is a pair $(R,R^+)$ consisting of a perfectoid $K$-algebra $R$ and an open and integrally closed subring $R^+\subset R^\circ$. In most cases, one will just take $R^+=R^\circ$. Then also the categories of perfectoid affinoid $K$-algebras and perfectoid affinoid $K^\flat$-algebras are equivalent. Huber associates to such pairs $(R,R^+)$ a topological space $X=\Spa(R,R^+)$ consisting of continuous valuations on $R$ that are $\leq 1$ on $R^+$, with the topology generated by the rational subsets $\{x\in X\mid \forall i: |f_i(x)|\leq |g(x)|\}$, where $f_1,\ldots,f_n,g\in R$ generate the unit ideal. Moreover, he defines a structure presheaf $\OO_X$ on $X$, and the subpresheaf $\OO_X^+$, consisting of functions which have absolute value $\leq 1$ everywhere.

\begin{thm} Let $(R,R^+)$ be a perfectoid affinoid $K$-algebra, with tilt $(R^\flat,R^{\flat +})$. Let $X=\Spa(R,R^+)$, with $\OO_X$, $\OO_X^+$, and $X^\flat = \Spa(R^\flat,R^{\flat +})$, with $\OO_{X^\flat}$, $\OO_{X^\flat}^+$.
\begin{altenumerate}
\item[{\rm (i)}] There is a canonical homeomorphism $X\cong X^\flat$, given by mapping $x$ to $x^\flat$ defined via $|f(x^\flat)| = |f^\sharp(x)|$. Rational subsets are identified under this homeomorphism.
\item[{\rm (ii)}] For any rational subset $U\subset X$, the pair $(\mathcal{O}_X(U),\mathcal{O}_X^+(U))$ is perfectoid affinoid with tilt $(\mathcal{O}_{X^\flat}(U),\mathcal{O}_{X^\flat}^+(U))$.
\item[{\rm (iii)}] The presheaves $\mathcal{O}_X$, $\mathcal{O}_X^+$ and $\OO_{X^\flat}$, $\OO_{X^\flat}^+$ are sheaves.
\item[{\rm (iv)}] For all $i>0$, the cohomology group $H^i(X,\mathcal{O}_X)=0$. Moreover, the cohomology group $H^i(X,\mathcal{O}_X^+)$ is almost zero, i.e. $\mathfrak{m}$-torsion. The same holds true for $\OO_{X^\flat}$ and $\OO_{X^\flat}^+$.
\end{altenumerate}
\end{thm}

This allows one to define general perfectoid spaces by gluing affinoid perfectoid spaces. Further, one can define \'etale morphisms of perfectoid spaces, and then \'etale topoi. This leads to an improvement on Faltings's almost purity theorem:

\begin{thm} Let $R$ be a perfectoid $K$-algebra, and let $S/R$ be finite etale. Then $S$ is perfectoid and $S^\circ$ is almost finite \'etale over $R^\circ$.
\end{thm}

In particular, no sort of semistable reduction hypothesis is required anymore. Also, the proof is much easier. In characteristic $p$, the result is easy, and is a consequence of repeated application of Frobenius (roughly, $S^\circ$ is finite \'etale over $R^\circ$ up to $t^N$-torsion for some $N$; but then it is finite \'etale up to $t^{N/p^k}$-torsion for all $k$, i.e. almost finite \'etale). Moreover, following the ideas for the case of perfectoid fields, one has a fully faithful functor from finite \'etale $R^\flat$-algebras to finite \'etale $R$-algebras. It is enough to prove that this functor is an equivalence of categories, as for finite \'etale $R$-algebras in the image, one can deduce the result from the case in characteristic $p$, which is already known. In order to prove that the functor is an equivalence of categories, one makes a localization argument on $X=\Spa(R,R^+)$ (for any $R^+\subset R^\circ$), to reduce to the case of perfectoid fields.

Tilting also identifies the \'etale topoi of a perfectoid space and its tilt, and as an application, one gets the following theorem.

\begin{thm} There is an equivalence of \'etale topoi of adic spaces
\[
(\mathbb{P}^n_{K^\flat})^\ad_\et\cong \varprojlim (\mathbb{P}^n_K)^\ad_\et\ .
\]
Here the transition maps are the $p$-th power map on homogeneous coordinates.
\end{thm}

In a sense, the theorem gives rise to a projection map
\[
\pi: \mathbb{P}^n_{K^\flat}\rightarrow \mathbb{P}^n_K
\]
defined on \'{e}tale topoi of adic spaces, which is given on coordinates by
\[
\pi(x_0:\ldots:x_n) = (x_0^\sharp:\ldots:x_n^\sharp)\ .
\]
It is also defined on the topological spaces underlying the adic spaces.

This discussion has the following consequence.

\begin{thm} Let $k$ be a finite extension of $\Q_p$, and let $X\subset \mathbb{P}^n_k$ be a smooth complete intersection. Then the weight-monodromy conjecture holds true for $X$.
\end{thm}

In \cite{ScholzePerfectoidSpaces1}, this result is proved for complete intersections in more general toric varieties in place of $\mathbb{P}^n$.

Let us first recall the statement of the weight-monodromy conjecture. Let $q$ be the cardinality of the residue field of $k$, and fix a geometric Frobenius $\Phi\in \Gal(\bar{k}/k)$, where $\bar{k}$ is a fixed geometric closure of $k$. The following are known about the \'etale cohomology group $V=H^i(X_{\bar{k}},\bar{\mathbb{Q}}_\ell)$, where $\ell\neq p$ is some prime, and $i\geq 0$.

\begin{altenumerate}
\item[{\rm (i)}] There is a decomposition
\[
V= \bigoplus_{j=0}^{2i} V_j\ ,
\]
where all eigenvalues of $\Phi$ on $V_j$ are Weil numbers of weight $j$, i.e. elements $\alpha\in \bar{\mathbb{Q}}$ such that $|\iota(\alpha)|=q^{j/2}$, under all embeddings $\iota:\bar{\mathbb{Q}}\to \mathbb{C}$.
\item[{\rm (ii)}] There is a nilpotent operator $N: V\to V$ given by the logarithm of the action of the $\ell$-adic inertia subgroup, inducing maps
\[
N: V_j\to V_{j-2}\ .
\]
\end{altenumerate}

Part (i) is a consequence of the Rapoport-Zink spectral sequence, \cite{RapoportZinkZeta}, in the case of semistable reduction, and de Jong's alterations, \cite{deJongAlterations}, in general, to reduce to this case. Part (ii) is a consequence of Grothendieck's quasi-unipotence theorem, which says that after a finite extension of $k$, the action of the inertia subgroup becomes unipotent. The decomposition in part (i) depends on the choice of $\Phi$, but the (weight) filtration given by $\Fil^m = \bigoplus_{j\leq m} V_j$ does not.

We note that if $k$ has good reduction, then by base-change, the action of $\Gal(\bar{k}/k)$ is unramified, so $N=0$, and by the Weil conjectures, $V=V_i$, i.e. all eigenvalues of the Frobenius are Weil numbers of weight $i$. In general, several different weights can appear, but this is supposed to be made good for by a nontrivial monodromy operator, connecting the different weights:

\begin{conj}[Weight-monodromy conjecture, \cite{DeligneICM}]\label{WMConj} For all $j\geq 0$, the map
\[
N^j: V_{i+j}\to V_{i-j}
\]
is an isomorphism.
\end{conj}

This is reminiscent of the Lefschetz decomposition. In fact, over $\C$, there is an analogue of this conjecture, by looking at a family of projective smooth complex varieties over the punctured unit disc. In that case, there is a limiting Hodge structure, which is not in general pure, and a monodromy operator (coming from a loop around the puncture). In that case, the result is known by work of Schmid, \cite{Schmid}, cf. also a paper of Steenbrink, \cite{Steenbrink} (who introduced what is now called the Rapoport-Zink spectral sequence).

Deligne, \cite{DeligneWeil2}, proved the analogue of Conjecture \ref{WMConj} for a finite extension of $\F_p((t))$ in place of $k$. Over $p$-adic fields, Conjecture \ref{WMConj} is known for $i=1$ (this reduces to abelian varieties, where one uses N\'eron models), and for $i=2$ (this reduces to surfaces, where it is proved in \cite{RapoportZinkZeta}). Apart from that, little is known.

Let us end this introduction by giving a short sketch of the proof of the weight-monodromy conjecture for a smooth hypersurface $X\subset \mathbb{P}^n$, which is already a new result; the proof in the general case is identical. Let $K$ be the completion of $k(\pi^{1/p^\infty})$, where $\pi$ is a uniformizer of $k$; this is a perfectoid field, and its tilt $K^\flat$ is the completion of $\F_q((t))(t^{1/p^\infty})$.\footnote{If $k=\Q_p$, then we are just back to our standard example.} Both the weights and the monodromy operator are defined by the action of $\Gal(\bar{K}/K)\subset \Gal(\bar{k}/k)$. We have the projection
\[
\pi: \mathbb{P}^n_{K^\flat}\rightarrow \mathbb{P}^n_K\ ,
\]
and we can look at the preimage $\pi^{-1}(X)$. One has a $\Gal(\bar{K}/K)\cong \Gal(\bar{K^\flat}/K^\flat)$-equivariant injective map $H^i(X)\rightarrow H^i(\pi^{-1}(X))$ on $\ell$-adic cohomology, and if $\pi^{-1}(X)$ were an algebraic variety, then one could deduce the result from Deligne's theorem in equal characteristic. However, the map $\pi$ is highly transcendental, and $\pi^{-1}(X)$ will not be given by equations. In general, it will look like some sort of fractal, have infinite-dimensional cohomology, and will have infinite degree in the sense that it will meet a line in infinitely many points. As an easy example, let
\[
X=\{x_0+x_1+x_2=0\}\subset (\mathbb{P}^2_K)^\ad\ .
\]
Then the homeomorphism
\[
|(\mathbb{P}^2_{K^\flat})^\ad|\cong \varprojlim_\varphi |(\mathbb{P}^2_K)^\ad|
\]
means that $\pi^{-1}(X)$ is topologically the inverse limit of the subvarieties
\[
X_n=\{x_0^{p^n}+x_1^{p^n}+x_2^{p^n}=0\}\subset (\mathbb{P}^2_K)^\ad\ .
\]
However, we have the following crucial approximation lemma.

\begin{lem} Let $\tilde{X}\subset (\mathbb{P}^n_K)^\ad$ be a small open neighborhood of the hypersurface $X$. Then there is a hypersurface $Y\subset \pi^{-1}(\tilde{X})$.
\end{lem}

The proof of this lemma is by an explicit approximation algorithm for the homogeneous polynomial defining $X$, and is the reason that we have to restrict to complete intersections. Using a result of Huber, one finds some $\tilde{X}$ such that $H^i(X) = H^i(\tilde{X})$, and hence gets a $\Gal(\bar{K}/K)\cong \Gal(\bar{K^\flat}/K^\flat)$-equivariant map $H^i(X)=H^i(\tilde{X})\rightarrow H^i(Y)$. As before, one checks (using $\cup$-products and a direct analysis in top degree) that it is injective and concludes.

\subsection{Some open questions}

Of course, the most urgent question would be to see whether the proof of the weight-monodromy conjecture can be extended to more general situations. The only problem is in proving an analogue of the last lemma; let me state this as a conjecture.

\begin{conj} Let $X\subset \mathbb{P}^n_K$ be a smooth closed subscheme. Let $\tilde{X}\subset (\mathbb{P}^n_K)^\ad$ be a small open neighborhood of $X$. Then there exists a closed subscheme $Y\subset \mathbb{P}^n_{K^\flat}$ such that $Y^\ad\subset \pi^{-1}(\tilde{X})$, and $\dim Y = \dim X$.
\end{conj}

If this conjecture is true, the weight-monodromy conjecture follows, using the same argument. We recall that $\pi^{-1}(X^\ad)\subset (\mathbb{P}^n_{K^\flat})^\ad$ can be regarded as a fractal. The question is whether this fractal can be approximated (globally) by an algebraic variety. We note that it can be approximated locally by an algebraic variety, as $X$ is smooth (and thus locally a complete intersection, so the argument in the case of global complete intersections works locally).

On the other hand, there are a couple of foundational questions on perfectoid spaces, where a positive answer would simplify many arguments. The key problem is whether the property of being perfectoid is local.

\begin{conj}\label{TechnicalConj} Let $K$ be some perfectoid field, $(A,A^+)$ a complete affinoid $K$-algebra. Assume that there is a cover of $X=\Spa(A,A^+)$ by rational subsets $U_i\subset X$ such that $\OO_X(U_i)$ is a perfectoid $K$-algebra. Then $A$ is a perfectoid $K$-algebra.
\end{conj}

\begin{cor} Assume Conjecture \ref{TechnicalConj}. Let $X$ be a perfectoid space over $K$, and assume that $U\subset X$ is affinoid, i.e. there is a complete affinoid $K$-algebra $(A,A^+)$ and a map
\[
(A,A^+)\to (H^0(U,\OO_X),H^0(U,\OO_X^+))
\]
such that the induced map $U\to \Spa(A,A^+)$ is a homeomorphism, and the map of (pre)sheaves $\OO_{\Spa(A,A^+)}\to \OO_U$ is an isomorphism locally on $U$.\footnote{This is saying that $U$ is an affinoid adic space in the sense of \cite{ScholzeWeinstein}.} Then
\[
(A,A^+)\cong (H^0(U,\OO_X),H^0(U,\OO_X^+))
\]
is an affinoid perfectoid $K$-algebra.
\end{cor}

\begin{cor} Assume Conjecture \ref{TechnicalConj}. Let $X$ be an affinoid perfectoid space over $K$, and let $Y$ be an affinoid noetherian adic space over $K$, and $f:X\to Y$ a map of adic spaces over $K$. Let $V$ be another affinoid noetherian adic space over $K$, with an \'etale map $g:V\to Y$. Then $U=V\times_Y X\to X$ is \'etale, and $U$ is affinoid perfectoid.
\end{cor}

\begin{rem} As the proof shows, one can also compute the global sections of $U$ in the expected way.
\end{rem}

\begin{proof} Let $X=\Spa(A,A^+)$, $Y=\Spa(B,B^+)$, $V=\Spa(C,C^+)$. Then let $D=A\otimes_B C$, with the topology making the image of $A_0\otimes_{B_0} C_0$ a ring of definition, with ideal of definition the image of $I\otimes J$, where $A_0\subset A$, $B_0\subset B$, $C_0\subset C$ are rings of definition, and $I\subset A_0$, $J\subset C_0$ are ideals of definition. Let $D^+\subset D$ be the integral closure of the image of $A^+\otimes_{B^+} C^+$. Then the completion of $(D,D^+)$ is an affinoid $K$-algebra, and it satisfies the hypothesis of Conjecture \ref{TechnicalConj}. For this, use that if $g: V\to Y$ is a composition of rational subsets and finite \'etale maps, then the completion of $(D,D^+)$ is a perfectoid affinoid $K$-algebra, with $\Spa(D,D^+) = U\subset X$ (as follows from the definition of the structure sheaf on rational subsets, resp. \cite{ScholzePerfectoidSpaces1}, Theorem 7.9 (iii)).
\end{proof}

Let us also note some subtleties related to saying that 'an inverse limit of adic spaces is perfectoid'. For this, we recall a definition from \cite{ScholzeWeinstein} in the case of interest here.

\begin{definition}\label{InverseLimit} Let $K$ be a perfectoid field, and let $X_i$, $i\in I$, be a cofiltered inverse system of locally noetherian adic spaces over $K$, with qcqs transition maps. Let $X$ be a perfectoid space over $K$ with a compatible system of maps $X\to X_i$, $i\in I$. Write
\[
X\sim \varprojlim_i X_i
\]
if $|X|\cong \varprojlim_i |X_i|$ is a homeomorphism, and there exists a cover of $X$ by open perfectoid affinoid $U=\Spa(R,R^+)$ for which the map
\[
\varinjlim_{\Spa(R_i,R_i^+)\subset X_i} R_i\to R
\]
has dense image, where the direct limit runs over all
\[
\Spa(R_i,R_i^+)\subset X_i
\]
over which $\Spa(R,R^+)\to X$ factors.
\end{definition}

\begin{rem} In \cite[Definition 2.4.1]{ScholzeWeinstein}, this notion is introduced for general adic spaces (in the sense of \cite[Definition 2.1.5]{ScholzeWeinstein}). The elementary properties of the following proposition also hold in that generality. It also follows from the arguments of the following proposition that the given definition is equivalent to the one in \cite{ScholzeWeinstein} in this situation: Note that here we require the open affinoid $U=\Spa(R,R^+)$ to come from a perfectoid affinoid $K$-algebra $(R,R^+)$, which would a priori lead to a stronger notion.
\end{rem}

Let us say that a perfectoid affinoid open subset $\Spa(R,R^+)\subset X$ is good if
\[
\varinjlim_{\Spa(R_i,R_i^+)\subset X_i} R_i\to R
\]
has dense image, and
\[
|\Spa(R,R^+)|\cong \varprojlim_{\Spa(R_i,R_i^+)\subset X_i} |\Spa(R_i,R_i^+)|
\]
is a homeomorphism.

\begin{prop}\begin{altenumerate}
\item[{\rm (i)}] If $X\sim \varprojlim_i X_i$, then there exists a cover of $X$ by perfectoid affinoid open $U=\Spa(R,R^+)\subset X$ which are good.
\item[{\rm (ii)}] If $U=\Spa(R,R^+)\subset X$ is good, then any rational subset of $U$ is good.
\end{altenumerate}
\end{prop}

\begin{proof}\begin{altenumerate}
\item[{\rm (i)}] Take any perfectoid affinoid $U=\Spa(R,R^+)\subset X$ such that
\[
\varinjlim_{\Spa(R_i,R_i^+)\subset X_i} R_i\to R
\]
has dense image. Clearly,
\[
|U|\subset \varprojlim_{\Spa(R_i,R_i^+)\subset X_i} |\Spa(R_i,R_i^+)|(\subset |X|)
\]
is open. Therefore, there is some qcqs open subset $U_i\subset \Spa(R_i,R_i^+)$ for some $i$, such that $U$ is the preimage of $U_i$. We may cover $U_i$ by rational subsets $U_i^\prime\subset \Spa(R_i,R_i^+)$; let $U^\prime\subset U$ be the preimage, which is a rational subset of $U$. Then
\[
|U^\prime| = \varprojlim_{\Spa(R_i^\prime,R_i^{\prime +})\subset X_i} |\Spa(R_i^\prime,R_i^{\prime +})|\ ,
\]
where now the inverse limit runs over all $\Spa(R_i^\prime,R_i^{\prime +})\subset X_i$ over which $U^\prime\to X_i$ factors. Moreover, the density statement is preserved (by the definition of the structure presheaf on rational subsets), so $U^\prime$ is good.
\item[{\rm (ii)}] A rational subset of $U$ comes as the preimage of a rational subset of some $\Spa(R_i,R_i^+)$; then use the argument from (i).
\end{altenumerate}
\end{proof}

\begin{prop}[{\cite[Proposition 2.4.5]{ScholzeWeinstein}}] In the situation of Definition \ref{InverseLimit}, the perfectoid space $X$ represents the functor
\[
Y\mapsto \varprojlim_i \Hom(Y,X_i)
\]
on perfectoid spaces over $K$. In particular, it is unique up to unique isomorphism.
\end{prop}

For the next proposition, one needs a slight variant of Conjecture \ref{TechnicalConj}.

\begin{conj}\label{TechnicalConj2} Let $K$ be some perfectoid field, and $(A,A^+)$ a complete affinoid $K$-algebra for which $A^+$ has the $p$-adic topology. Assume that there is a covering of $X=\Spa(A,A^+)$ by rational subsets $U_i\subset X$ for which $\widehat{\OO_X(U_i)}$ is a perfectoid $K$-algebra, where the completion is taken with respect to the topology giving $\OO_X^+(U_i)$ the $p$-adic topology. Then $(A,A^+)$ is a perfectoid affinoid $K$-algebra.
\end{conj}

\begin{rem} It is not clear whether the natural topology on $\OO_X^+(U_i)$ is the $p$-adic topology, even if $A^+$ has the $p$-adic topology. If $(A,A^+)$ is perfectoid, this is true.
\end{rem}

\begin{prop} Assume Conjecture \ref{TechnicalConj2}. Assume that all $X_i=\Spa(R_i,R_i^+)$ are affinoid, and that $X\sim \varprojlim_i X_i$. Then $X$ is good, i.e. $X=\Spa(R,R^+)$ is affinoid perfectoid and $\varinjlim_i R_i\to R$ has dense image.
\end{prop}

\begin{proof} As before: One knows the result after localizing to rational subsets, so Conjecture \ref{TechnicalConj2} gives the result (using the completion of the direct limit of the $(R_i,R_i^+)$, with the $p$-adic topology on $\varinjlim R_i^+$, as $(A,A^+)$). Note that we have to enforce artificially the $p$-adic topology on $A^+$, and we have to do the same on rational subsets, which is the reason that we need Conjecture \ref{TechnicalConj2} in place of Conjecture \ref{TechnicalConj}.
\end{proof}
\newpage

\section{$p$-adic Hodge theory}

\subsection{Introduction}

This introduction is essentially identical to the introduction to \cite{ScholzeHodge}.

In the paper \cite{ScholzeHodge}, we started to investigate to what extent $p$-adic comparison theorems stay true for rigid-analytic varieties. Up to now, such comparison isomorphisms were mostly studied for schemes over $p$-adic fields, but we show there that the whole theory extends naturally to rigid-analytic varieties over $p$-adic fields. This is of course in analogy with classical Hodge theory, which most naturally is formulated in terms of complex-analytic spaces.

Several difficulties have to be overcome to make this work. The first is that finiteness of $p$-adic \'{e}tale cohomology is not known for rigid-analytic varieties over $p$-adic fields. In fact, it is false if one does not make a restriction to the proper case. However, we prove that for proper smooth rigid-analytic varieties, finiteness of $p$-adic \'{e}tale cohomology holds.

\begin{thm}\label{Thm1} Let $C$ be a complete algebraically closed extension of $\Q_p$, let $X/C$ be a proper smooth rigid-analytic variety, and let $\mathbb{L}$ be an $\mathbb{F}_p$-local system on $X_\et$. Then $H^i(X_\et,\mathbb{L})$ is a finite-dimensional $\mathbb{F}_p$-vector space for all $i\geq 0$, which vanishes for $i>2\dim X$.
\end{thm}

The properness assumption is crucial here; the smoothness assumption is in fact unnecessary, and an artefact of the proof -- using resolution of singularities, one can deduce the result for general proper rigid-analytic varieties, see Theorem \ref{NonSmooth} below. We note that in the smooth case, it would be interesting to prove Poincar\'{e} duality.

Let us first explain our proof of this theorem. We build upon Faltings's theory of almost \'{e}tale extensions, amplified by the theory of perfectoid spaces. One important difficulty in $p$-adic Hodge theory as compared to classical Hodge theory is that the local structure of rigid-analytic varieties is very complicated; small open subsets still have a large \'{e}tale fundamental group. We introduce the pro-\'{e}tale site $X_\proet$ whose open subsets are roughly of the form $V\rightarrow U\rightarrow X$, where $U\rightarrow X$ is some \'{e}tale morphism, and $V\rightarrow U$ is an inverse limit of finite \'{e}tale maps. Then the local structure of $X$ in the pro-\'{e}tale topology is simpler, namely, it is locally perfectoid. This amounts to extracting lots of $p$-power roots of units in the tower $V\rightarrow U$. We note that the idea to extract many $p$-power roots is common to all known proofs of comparison theorems in $p$-adic Hodge theory.

The following result gives further justification to the definition of pro-\'etale site.

\begin{thm}\label{ThmKpi1} Let $X$ be a connected affinoid rigid-analytic variety over $C$. Then $X$ is a $K(\pi,1)$ for $p$-torsion coefficients, i.e. for all $p$-torsion local systems $\mathbb{L}$ on $X$, the natural map
\[
H^i_\cont(\pi_1(X,x),\mathbb{L}_x)\to H^i(X_\et,\mathbb{L})
\]
is an isomorphism. Here, $x\in X(C)$ is a base point, and $\pi_1(X,x)$ denotes the profinite \'etale fundamental group.
\end{thm}

We note that we assume only that $X$ is affinoid;  no smallness or nonsingularity hypothesis is necessary for this result. This theorem implies that $X$ is 'locally contractible' in the pro-\'etale site, at least for $p$-torsion local systems.

Now, on affinoid perfectoid subsets $U$, one knows that $H^i(U_\et,\OO_X^+/p)$ is almost zero for $i>0$, where $\OO_X^+\subset \OO_X$ is the subsheaf of functions of absolute value $\leq 1$ everywhere. This should be seen as the basic finiteness result, and is related to Faltings's almost purity theorem. Starting from this and a suitable cover of $X$ by affinoid perfectoid subsets in $X_\proet$, one can deduce that $H^i(X_\et,\OO_X^+/p)$ is almost finitely generated over $\OO_K$. At this point, one uses that $X$ is proper, and in fact the proof of this finiteness result is inspired by the proof of finiteness of coherent cohomology of proper rigid-analytic varieties, as given by Kiehl, \cite{KiehlFiniteness}. Then one deduces finiteness results for the $\mathbb{F}_p$-cohomology by using a variant of the Artin-Schreier sequence
\[
0\rightarrow \mathbb{F}_p\rightarrow \OO_X^+/p\rightarrow \OO_X^+/p\rightarrow 0\ .
\]
In fact, the proof shows at the same time the following result, which is closely related to \S3, Theorem 8, of \cite{FaltingsAlmostEtale}.

\begin{thm}\label{Thm2} In the situation of Theorem \ref{Thm1}, there is an almost isomorphism of $\OO_C$-modules for all $i\geq 0$,
\[
H^i(X_\et,\mathbb{L})\otimes \OO_C/p \rightarrow H^i(X_\et,\mathbb{L}\otimes \OO_X^+/p)\ .
\]
More generally, assume that $f:X\rightarrow Y$ is a proper smooth morphism of rigid-analytic varieties over $C$, and $\mathbb{L}$ is an $\mathbb{F}_p$-local system on $X_\et$. Then there is an almost isomorphism for all $i\geq 0$,
\[
(R^if_{\et\ast} \mathbb{L})\otimes \OO_Y^+/p\rightarrow R^if_{\et\ast}(\mathbb{L}\otimes \OO_X^+/p)\ .
\]
\end{thm}

\begin{rem} The relative case was already considered in an appendix to \cite{FaltingsAlmostEtale}: Under the assumption that $X$, $Y$ and $f$ are algebraic and have suitable integral models, this is \S6, Theorem 6, of \cite{FaltingsAlmostEtale}. In our approach, it is a direct corollary of the absolute version.
\end{rem}

In a sense, this can be regarded as a primitive version of a comparison theorem. It has the very interesting (and apparently paradoxical) feature that on the right-hand side, one has a 'coherent cohomology group modulo $p$', but the cohomology is computed on the generic fibre. That this group behaves well relies on the fact that perfectoid spaces have a canonical 'almost integral' structure. Moreover, it implies the following strange property of
\[
H^i(X_\et,\OO_X^+)\ :
\]
After inverting $p$, these groups are usual coherent cohomology $H^i(X_\et,\OO_X) = H^i(X_\an,\OO_X)$, but after modding out $p$, one gets \'etale cohomology. Thus, these integral cohomology groups build a bridge between \'etale and coherent cohomology.

Although it should be possible to deduce (log-)crystalline comparison theorems from here, we did only the de Rham case. For this, we introduced sheaves on $X_\proet$, which we call period sheaves, as their values on pro-\'{e}tale covers of $X$ give period rings. Among them is the sheaf $\B_\dR^+$, which is the relative version of Fontaine's ring $B_\dR^+$. Any lisse $\Z_p$-sheaf $\mathbb{L}$ on $X_\et$ gives rise to a $\B_\dR^+$-local system on $X_\proet$, and it is a formal consequence of Theorem \ref{Thm2} that
\begin{equation}\label{NotProvedByBeilinson}
H^i_\et(X,\mathbb{L})\otimes_{\Z_p} B_\dR^+\cong H^i(X_\proet,\mathbb{M})\ .
\end{equation}

We want to compare this to de Rham cohomology. For this, we first relate filtered modules with integrable connection to $\B_\dR^+$-local systems.

\begin{thm}\label{Thm3} Let $X$ be a smooth rigid-analytic variety over $k$, where $k$ is a complete discretely valued nonarchimedean extension of $\Q_p$ with perfect residue field. Then there is a fully faithful functor from the category of filtered $\OO_X$-modules with an integrable connection satisfying Griffiths transversality, to the category of $\B_\dR^+$-local systems.
\end{thm}

The proof makes use of the period rings introduced in Brinon's book \cite{BrinonRepresentations}, and relies on some of the computations of Galois cohomology groups done there. We say that a lisse $\Z_p$-sheaf $\mathbb{L}$ is de Rham if the associated $\B_\dR^+$-local system $\mathbb{M}$ lies in the essential image of this functor. We get the following comparison result.

\begin{thm}\label{Thm4} Let $k$ be a discretely valued complete nonarchimedean extension of $\Q_p$ with perfect residue field $\kappa$, and algebraic closure $\bar{k}$, and let $X$ be a proper smooth rigid-analytic variety over $k$. For any lisse $\Z_p$-sheaf $\mathbb{L}$ on $X_\et$ with associated $\B_\dR^+$-local system $\mathbb{M}$, we have a $\Gal(\bar{k}/k)$-equivariant isomorphism
\[
H^i_\et(X_{\bar{k}},\mathbb{L})\otimes_{\Z_p} B_\dR^+\cong H^i(X_{\bar{k},\proet},\mathbb{M})\ .
\]
If $\mathbb{L}$ is de Rham, with associated filtered module with integrable connection $(\mathcal{E},\nabla,\Fil^\bullet)$, then the Hodge-de Rham spectral sequence
\[
H^{i-j,j}_{\mathrm{Hodge}}(X,\mathcal{E})\Rightarrow H^i_\dR(X,\mathcal{E})
\]
degenerates. Moreover, $H^i_\et(X_{\bar{k}},\mathbb{L})$ is a de Rham representation of $\Gal(\bar{k}/k)$ with associated filtered $k$-vector space $H^i_\dR(X,\mathcal{E})$. In particular, there is also a $\Gal(\bar{k}/k)$-equivariant isomorphism
\[
H^i_\et(X_{\bar{k}},\mathbb{L})\otimes_{\Z_p} \hat{\bar{k}}\cong \bigoplus_j H^{i-j,j}_{\mathrm{Hodge}}(X,\mathcal{E})\otimes_k \hat{\bar{k}}(-j)\ .
\]
\end{thm}

\begin{rem} We define the Hodge cohomology as the hypercohomology of the associated gradeds of the de Rham complex of $\mathcal{E}$, with the filtration induced from $\Fil^\bullet$.
\end{rem}

In particular, we get the following corollary, which answers a question of Tate, \cite{TatePDivGroups}, Remark on p.180.

\begin{cor} For any proper smooth rigid-analytic variety $X$ over $k$, the Hodge-de Rham spectral sequence
\[
H^i(X,\Omega_X^j)\Rightarrow H^{i+j}_\dR(X)
\]
degenerates, there is a Hodge-Tate decomposition
\[
H^i_\et(X_{\bar{k}},\Q_p)\otimes_{\Q_p} \hat{\bar{k}}\cong \bigoplus_{j=0}^i H^{i-j}(X,\Omega_X^j)\otimes_k \hat{\bar{k}}(-j)\ ,
\]
and the $p$-adic \'{e}tale cohomology $H^i_\et(X_{\bar{k}},\Q_p)$ is de Rham, with associated filtered $k$-vector space $H^i_\dR(X)$.
\end{cor}

Interestingly, no 'K\"ahler' assumption is necessary for this result in the $p$-adic case as compared to classical Hodge theory. In particular, one gets degeneration for all proper smooth varieties over fields of characteristic $0$ without using Chow's lemma.

Examples of non-algebraic proper smooth rigid-analytic varieties can be constructed by starting from a proper smooth variety in characteristic $p$, and taking a formal, non-algebraizable, lift to characteristic $0$. This can be done for example for abelian varieties or K3 surfaces. More generally, there is the theory of abeloid varieties, which are 'non-algebraic abelian rigid-analytic varieties', roughly, cf. \cite{LuetkebohmertAbeloid}.

There are also some non-K\"ahler compact complex manifolds over $\C$ which have $p$-adic analogues: Fortunately, only those for which Hodge-de Rham degeneration holds. An example is the Hopf surface. Over a $p$-adic field, this can be defined as follows. Fix an element $q\in k$ with $|q|<1$, and let
\[
X = (\mathbb{A}^2\setminus \{(0,0)\}) / q^\Z\ ,
\]
where $q$ acts by diagonal multiplication. It is easy to see that $X$ is proper and smooth. It has $H^0(X,\Omega_X^1) = 0$, $H^1(X,\OO_X) = k$ (so Hodge symmetry fails!), and $H^1_\et(X,\Z_\ell) = \Z_\ell$ for any prime number $\ell$ (including $\ell=p$). In particular, the weight-monodromy conjecture fails badly for this non-algebraic variety. However, one may formulate the following conjecture on independence of $\ell$:

\begin{conj} Let $X$ be a proper smooth rigid-analytic variety over a finite extension $k$ of $\Q_p$. Then the Weil-Deligne representation associated to $H^i_\et(X_{\bar{k}},\Q_\ell)$ is independent of $\ell$ (including $\ell=p$).
\end{conj}

For $\ell\neq p$, we take the usual recipee, and for $\ell=p$, we use Fontaine's recipee, using that $H^i_\et(X_{\bar{k}},\Q_\ell)$ is de Rham (and thus potentially semistable).

Theorem \ref{Thm4} also has the following consequence, which was conjectured by Schneider, cf. \cite{SchneiderLocalSystems}, p.633.

\begin{cor} Let $k$ be a finite extension of $\Q_p$, let $X=\Omega^n_k$ be Drinfeld's upper half-space, which is the complement of all $k$-rational hyperplanes in $\mathbb{P}^{n-1}_k$, and let $\Gamma\subset \PGL_n(k)$ be a discrete cocompact subgroup acting without fixed points on $\Omega^n_k$. One gets the quotient $X_\Gamma = X/\Gamma$, which is a proper smooth rigid-analytic variety over $k$. Let $M$ be a representation of $\Gamma$ on a finite-dimensional $k$-vector space, such that $M$ admits a $\Gamma$-invariant $\OO_k$-lattice. It gives rise to a local system $\mathcal{M}_\Gamma$ of $k$-vector spaces on $X_\Gamma$. Then the twisted Hodge-de Rham spectral sequence
\[
H^i(X_\Gamma,\Omega_{X_\Gamma}^j\otimes \mathcal{M}_\Gamma)\Rightarrow H^{i+j}_\dR(X_\Gamma,\OO_{X_\Gamma}\otimes \mathcal{M}_\Gamma)
\]
degenerates.
\end{cor}

The proof of Theorem \ref{Thm4} follows the ideas of Andreatta and Iovita, \cite{AndreattaIovita}, in the crystalline case. One uses a version of the Poincar\'{e} lemma, which says here that one has an exact sequence of sheaves over $X_\proet$,
\[
0\rightarrow \B_\dR^+\rightarrow \OO\B_\dR^+\buildrel\nabla\over\rightarrow \OO\B_\dR^+\otimes_{\OO_X}\Omega_X^1\buildrel\nabla\over\rightarrow\ldots\ ,
\]
where we use slightly nonstandard notation. In \cite{BrinonRepresentations} and \cite{AndreattaIovita}, $\B_\dR^+$ would be called $\B_\dR^{\nabla +}$, and $\OO\B_\dR^+$ would be called $\B_\dR^+$. This choice of notation is used because many sources do not consider sheaves like $\OO\B_\dR^+$, and agree with our notation in writing $\B_\dR^+$ for the sheaf that is sometimes called $\B_\dR^{\nabla +}$.

Given this Poincar\'{e} lemma, it only remains to calculate the cohomology of $\OO\B_\dR^+$, which turns out to be given by coherent cohomology through some explicit calculation. This finishes the proof of Theorem \ref{Thm4}. We note that this proof is direct: All desired isomorphisms are proved by a direct argument, and not by producing a map between two cohomology theories and then proving that it has to be an isomorphism by abstract arguments. In fact, such arguments would not be available for us, as results like Poincar\'{e} duality are not known for the $p$-adic \'{e}tale cohomology of rigid-analytic varieties over $p$-adic fields. It also turns out that our methods are flexible enough to handle the relative case, and our results imply directly the corresponding results for proper smooth algebraic varieties, by suitable GAGA results. This gives for example the following result. We should note that this is the first general relative de Rham comparison result, even in the algebraic case.

\begin{thm}\label{Thm5} Let $k$ be a discretely valued complete nonarchimedean extension of $\Q_p$ with perfect residue field $\kappa$, and let $f:X\rightarrow Y$ be a proper smooth morphism of smooth rigid-analytic varieties over $k$. Let $\mathbb{L}$ be a lisse $\Z_p$-sheaf on $X_\et$ which is de Rham, with associated filtered module with integrable connection $(\mathcal{E},\nabla,\Fil^\bullet)$. Assume that $R^if_\et \mathbb{L}$ is a lisse $\Z_p$-sheaf on $Y_\et$; this holds true, for example, if the situation comes as the analytification of algebraic objects.

Then $R^if_\et \mathbb{L}$ is de Rham, with associated filtered module with integrable connection given by $R^if_{\dR\ast} (\mathcal{E},\nabla,\Fil^\bullet)$.
\end{thm}

We note that we make use of the full strength of the theory of perfectoid spaces. Apart from this, our argument is rather elementary and self-contained, making use of little more than basic rigid-analytic geometry, which we reformulate in terms of adic spaces, and basic almost mathematics. In particular, we work entirely on the generic fibre. This eliminates in particular any assumptions on the reduction type of our variety, and we do not need any version of de Jong's alterations, neither do we need log structures. The introduction of the pro-\'{e}tale site makes all constructions functorial, and it also eliminates the need to talk about formal projective or formal inductive systems of sheaves, as was done e.g. in \cite{FaltingsAlmostEtale}, \cite{AndreattaIovita}: All period sheaves are honest sheaves on the pro-\'{e}tale site.

\subsection{A comparison result for constructible coefficients}

The methods of \cite{ScholzeHodge} have some consequences that were not included there. First, we explain how to deduce a comparison result in the style of Theorem \ref{Thm2} with constructible coefficients.

Let $C$ be a complete algebraically closed extension of $\Q_p$, with a fixed open valuation subring $C^+\subset C$ (e.g. $C^+ = \OO_C$), and let $f: X\to Y$ be a proper map of schemes of finite type over $C$, with associated adic spaces $f^\ad: X^\ad\to Y^\ad$ over $\Spa(C,C^+)$. Let $\mathbb{L}$ be a constructible $\F_p$-sheaf on $X$, with pullback $\mathbb{L}^\ad$ to $X^\ad$. Recall the following result of Huber, \cite{Huber}, Theorem 3.7.2.

\begin{thm} For all $i\geq 0$, $(R^if_{\et\ast} \mathbb{L})^\ad\buildrel\cong\over\longrightarrow R^if^\ad_{\et\ast} \mathbb{L}^\ad$.
\end{thm}

In this section, we prove the following constructible version of Theorem \ref{Thm2}.

\begin{thm}\label{ConstrComp} For all $i\geq 0$, the map $(R^if^\ad_{\et\ast}\mathbb{L}^\ad)\otimes \OO_{Y^\ad}^+/p\to R^if^\ad_{\et\ast}(\mathbb{L}^\ad\otimes \OO_{X^\ad}^+/p)$ is an almost isomorphism.
\end{thm}

We remark that if $Y=\Spec C$ (i.e., in the absolute context), this says that
\[
H^i(X_\et,\mathbb{L})\otimes C^+/p\cong H^i(X^\ad_\et,\mathbb{L}^\ad)\otimes C^+/p\to H^i(X^\ad_\et,\mathbb{L}^\ad\otimes \OO_{X^\ad}^+/p)
\]
is an almost isomorphism. We need a simple lemma.

\begin{lem} Let $X$ be a locally noetherian adic space over $\Spa(\Q_p,\Z_p)$, with closed subspace $i:Z\hookrightarrow X$.
\begin{altenumerate}
\item[{\rm (i)}] The map $i^\ast \OO_X^+/p\to \OO_Z^+/p$ is an isomorphism (on $Z_\an$, and on $Z_\et$).
\item[{\rm (ii)}] For any $\F_p$-sheaf $F$ on $Z_\et$,
\[
i_\ast F\otimes \OO_X^+/p\buildrel\cong\over\longrightarrow i_\ast(F\otimes \OO_Z^+/p)\ .
\]
\end{altenumerate}
\end{lem}

\begin{proof}\begin{altenumerate}
\item[{\rm (i)}] Let $\overline{f}\in \OO_Z^+/p$ be a section; take a lift $f\in \OO_Z^+$; by approximation, we may assume that $f$ is the image of some $g\in \OO_X$. The locus $|g|\leq 1$ in $X$ is open; thus, we may assume $g\in \OO_X^+$. This shows that $i^\ast \OO_X^+/p\to \OO_Z^+/p$ is surjective. Moreover, if $g\in \OO_X^+$ that becomes divisible by $p$ on $Z$, look at the open locus $|g|\leq |p|$ in $X$: Again, it contains $Z$, which shows that $g$ becomes $0$ in $i^\ast \OO_X^+/p$.
\item[{\rm (ii)}] We check fibres at all geometric points. At fibres outside $Z$, both are zero. At fibres in $Z$, it follows from (i), noting that we may ignore $i_\ast$ then, and that taking fibres commutes with tensor products.
\end{altenumerate}
\end{proof}

We recall from Theorem \ref{Thm2} that Theorem \ref{ConstrComp} is true when $X$ is smooth and $\mathbb{L}$ is trivial (or just locally constant). From here, we get another base case.

\begin{lem}\label{ConstrCompSmooth} Assume that $X$ is smooth, and let $D=\bigcup_a D_a\subset X$, $a\in I$, be a simple normal crossings divisor with smooth irreducible components $D_a\subset X$. Let $U=X\setminus D$, with open embedding $j:U\to X$. Then Theorem \ref{ConstrComp} holds true for $\mathbb{L} = j_! \F_p$ and $Y=\Spec C$ (i.e., in the absolute context).
\end{lem}

\begin{proof} For $J\subset I$, let $D_J=\bigcap_{a\in J} D_a$, and set $D_\emptyset = X$. Then all $D_J$ are proper and smooth. We have the closed embeddings $i_J: D_J\to X$, and a long exact sequence
\[
0\to j_! \F_p\to \F_p\to \bigoplus_a i_{a\ast} \F_p\to \ldots\to \bigoplus_{|J|=k} i_{J\ast} \F_p\to \ldots \to i_{I\ast} \F_p\to 0
\]
of sheaves on $X_\et$. By pullback, we get a similar exact sequence on $X^\ad_\et$. Now, we tensor with $\OO_{X^\ad}^+/p$ and get a similar long exact sequence
\[
0\to j_! \F_p\otimes \OO_{X^\ad}^+/p\to \OO_{X^\ad}^+/p\to \bigoplus_a i_{a\ast} \F_p\otimes \OO_{X^\ad}^+/p\to \ldots\ .
\]
But $i_{J\ast} \F_p\otimes \OO_{X^\ad}^+/p = i_{J\ast} \OO_{D_J^\ad}^+/p$. Thus, the lemma follows upon applying the result for all $D_J$, with trivial coefficients.
\end{proof}

\begin{lem}\label{SmoothCase} Let $Y=\Spec C$, and $X$ proper over $Y$. Let $j:U\hookrightarrow X$ be a smooth dense open subscheme. Then Theorem \ref{ConstrComp} holds true for $\mathbb{L} = j_! \F_p$.
\end{lem}

\begin{proof} Take a resolution of singularities $f: \tilde{X}\to X$, which is an isomorphism above $U$, and such that the boundary $D=\tilde{X}\setminus U\subset \tilde{X}$ is a divisor with simple normal crossings. Let $\tilde{j}: U\hookrightarrow \tilde{X}$ denote the lifted embedding. The result follows from the previous lemma, together with the simple observation
\[
j_!\F_p\otimes \OO_{X^\ad}^+/p\cong Rf^\ad_{\et\ast} (\tilde{j}_! \F_p\otimes \OO_{\tilde{X}^\ad}^+/p)\ .
\]
This may be checked on fibres at points. Over $U$, it is clear, and outside $U$, everything vanishes. (Use Proposition 2.6.1 of \cite{Huber} to compute the fibre of the pushforward as the cohomology of the fibre.)
\end{proof}

\begin{proof} {\it (of Theorem \ref{ConstrComp}.)} As in \cite{ScholzeHodge}, proof of Corollary 5.11, the relative version reduces immediately to the absolute version, so we may assume that $Y=\Spec C$, $Y^\ad = \Spa(C,C^+)$. (This reduction step is the reason that we allow general open valuation rings $C^+\subset C$ from the start.)

We argue by induction on $\dim X$. So assume the statement is known in dimension $<n$, let $X$ be of dimension $n$, and let $\mathbb{L}$ be some constructible $\F_p$-sheaf on $X$. We may assume that $X$ is reduced. Then there is a smooth dense open subscheme $j:U\hookrightarrow X$ such that $j^\ast \mathbb{L}$ is locally constant; let $i:Z\hookrightarrow X$ be the closed complement, so that $\dim Z < n$. We have a short exact sequence
\[
0\to j_! j^\ast \mathbb{L}\to \mathbb{L}\to i_\ast i^\ast \mathbb{L}\to 0\ ,
\]
which gives rise to a short exact sequence
\[
0\to j_! j^\ast \mathbb{L}^\ad\otimes \OO_{X^\ad}^+/p\to \mathbb{L}^\ad\otimes \OO_{X^\ad}^+/p\to (i_\ast i^\ast \mathbb{L}^\ad)\otimes \OO_{X^\ad}^+/p\to 0\ .
\]
But $(i_\ast i^\ast \mathbb{L}^\ad)\otimes \OO_{X^\ad}^+/p = i_\ast(i^\ast \mathbb{L}^\ad\otimes \OO_{Z^\ad}^+/p)$, so the isomorphism for the right-hand term follows by induction from the assertion for $Z$ and $i^\ast \mathbb{L}$. Thus, we may assume that there is a smooth dense open subscheme $j: U\hookrightarrow X$ and a locally constant sheaf $F$ on $U$ such that $\mathbb{L} = j_! F$.

The case that $F$ is trivial has been handled in Lemma \ref{ConstrCompSmooth}. In particular, this handles the case $n=0$ of relative dimension $0$.

In the general case, let $V\to U$ be a finite \'etale Galois cover over which $F$ becomes trivial, and let $Y\to X$ be the normalization of $X$ in $V$. Let $Y_k$ be the $k$-fold fibre product of $Y$ over $X$, with open subset $V_k\subset Y_k$, which is the $k$-fold fibre product of $V$ over $U$. Let $j_k: V_k\to Y_k$ be the open embedding, and $f_k:Y_k\to X$, $g_k:V_k\to U$ the projections. We have a resolution of $j_! F$ on $X_\et$:
\[
0\to j_! F\to f_{1\ast} j_{1!} g_1^\ast F\to f_{2\ast} j_{2!} g_2^\ast F\to \ldots\ .
\]
This induces a similar resolution on $X^\ad_\et$. Similarly, we have a resolution of $(j_! F\otimes \OO_{X^\ad}^+/p)^a$ on $X^\ad_\et$:
\[
0\to ((j_! F)^\ad\otimes \OO_{X^\ad}^+/p)^a\to f^\ad_{1\ast} ((j_{1!} g_1^\ast F)^\ad\otimes \OO_{Y_1^\ad}^+/p)^a\to f^\ad_{2\ast} ((j_{2!} g_2^\ast F)^\ad\otimes \OO_{Y_2^\ad}^+/p)^a\to \ldots\ :
\]
To see this, we may check on fibres. Away from $U$, everything is $0$. At a geometric point $\bar{x}\to U$ with values in $(L,L^+)$, and fibre $Y_{k\bar{x}}\subset Y_k$, the sequence identifies with
\[
0\to F_{\bar{x}}\otimes L^{+a}/p\to H^0(Y_{1\bar{x}},F_{\bar{x}})\otimes L^{+a}/p\to H^0(Y_{2\bar{x}},F_{\bar{x}})\otimes L^{+a}/p\to \ldots
\]
(by the result in relative dimension $0$). The exactness of this sequence is just the exactness of
\[
0\to j_! F\to f_{1\ast} j_{1!} g_1^\ast F\to f_{2\ast} j_{2!} g_2^\ast F\to \ldots
\]
at $\bar{x}$, tensored with $L^{+a}/p$.

Finally, we get the result by applying it for all $j_{k!} g_k^\ast F$ on $Y_k$, noting that $g_k^\ast F$ is trivial.
\end{proof}

At this point, let us also remark that by using resolution of singularities for rigid-analytic varieties\footnote{See the paper by Bierstone-Milman \cite{BierstoneMilman}, situation (0.1) (2), where they state that it works for (good) Berkovich spaces. But proper rigid-analytic varieties, proper adic spaces and proper Berkovich spaces are canonically equivalent, and proper Berkovich spaces are good by \cite{Temkin}, so one may use their result.}, one gets Theorem \ref{Thm1} for non-smooth spaces by the same argument as above.

\begin{thm}\label{NonSmooth} Let $C$ be an algebraically closed complete extension of $\Q_p$, and let $X/C$ be a proper rigid-analytic variety. Then $H^i(X_\et,\F_p)$ is a finite-dimensional $\F_p$-vector space for all $i\geq 0$, which vanishes for $i>2\dim X$. Moreover, there is an almost isomorphism
\[
H^i(X_\et,\F_p)\otimes_{\F_p} \OO_C/p\to H^i(X_\et,\OO_X^+/p)
\]
for all $i\geq 0$.
\end{thm}

\begin{proof} We use the same arguments as above. We claim more generally that if $U=X\setminus Z$ is the complement of a Zariski closed subset $Z$ in a proper rigid-analytic variety $X$, then $H^i(X_\et,j_! \F_p)$ is a finite-dimensional $\F_p$-vector space, which vanishes for $i>2\dim X$, and with a similar almost isomorphism. By induction on $\dim X$, the result is known for $Z$, and the result for $X$ is equivalent to the result for $U$. Given $X$, we may thus assume that $U$ is smooth. By resolution of singularities, we may then assume that $X$ is smooth, and $Z\subset X$ is a divisor with normal crossings (cf. proof of Lemma \ref{SmoothCase}). In that case, one argues as in the proof of Lemma \ref{ConstrCompSmooth}.
\end{proof}

\subsection{The Hodge-Tate spectral sequence}

Let $C$ be an algebraically closed complete extension of $\Q_p$, and let $X/C$ be a proper smooth rigid-analytic variety. There is a general Hodge-Tate spectral sequence. It is instructive to compare it to the Hodge-de Rham spectral sequence:

\begin{thm} There is a Hodge-de Rham spectral sequence
\[
E_1^{ij} = H^j(X,\Omega_X^i)\Rightarrow H^{i+j}_\dR(X)\ .
\]
\end{thm}

\begin{rem} If $X$ is a scheme, or $X$ is defined over a discretely valued subfield $K\subset C$ with perfect residue field, then this sequence degenerates (by the Lefschetz principle, resp. our result stated above). We conjecture that it degenerates in general; however, our methods are not sufficient to prove this. Assuming that it degenerates, one gets a decreasing Hodge-de Rham filtration $\Fil^\bullet H^i_\dR(X)$, with
\[
\Fil^q H^i_\dR(X) / \Fil^{q+1} H^i_\dR(X) = H^{i-q}(X,\Omega_X^q)\ .
\]
\end{rem}

\begin{thm}\label{HodgeTateSpecSeq} There is a Hodge-Tate spectral sequence
\[
E_2^{ij} = H^i(X,\Omega_X^j)(-j)\Rightarrow H^{i+j}_\et(X,\Q_p)\otimes_{\Q_p} C\ .
\]
\end{thm}

\begin{rem} Again, if $X$ is a scheme, or $X$ is defined over a discretely valued subfield $K\subset C$ with perfect residue field (more generally, if $C(-j)^{\Gal(\bar{K}/K)}=0$ for $j\neq 0$), then this spectral sequence degenerates. For schemes, this follows by a dimension count (noting that by the Lefschetz principle, \'etale and de Rham cohomology have the same dimension); in the other case, it follows because the differentials are $\Gal(\bar{K}/K)$-equivariant. Note that in the case of schemes, one could also use a spreading-out argument to reduce to the case of a discretely valued subfield $K\subset C$ with perfect residue field; the same argument could be applied in the Hodge-de Rham case. This gives a purely $p$-adic proof of these results.

In case the Hodge-Tate spectral sequence degenerates, one gets a decreasing Hodge-Tate filtration $\Fil^\bullet (H^i_\et(X,\Q_p)\otimes_{\Q_p} C)$, with
\[
\Fil^q (H^i_\et(X,\Q_p)\otimes_{\Q_p} C) / \Fil^{q+1} (H^i_\et(X,\Q_p)\otimes_{\Q_p} C) = H^q(X,\Omega_X^{i-q})(q-i)\ .
\]
\end{rem}

\begin{rem} The Hodge-Tate spectral sequence does not have a direct analogue over the complex numbers $\mathbb{C}$. Note that in the $p$-adic case, $H^i_\et(X,\Q_p)\otimes_{\Q_p} C$ is not canonically isomorphic to $H^i_\dR(X)$; thus the two spectral sequences do not converge to the same cohomology groups. Over $\mathbb{C}$, the Hodge-de Rham filtration is canonically split; in the $p$-adic case, however, both the Hodge-de Rham and the Hodge-Tate filtration are not canonically split.
\end{rem}

\begin{proof}{\it (of Theorem \ref{HodgeTateSpecSeq}.)} We use our results from \cite{ScholzeHodge}. First, Theorem \ref{Thm2} says that
\[
H^i(X_\et,\F_p)\otimes_{\F_p} \OO_C/p\to H^i(X_\et,\OO_X^+/p)
\]
is an almost isomorphism. By induction on $n$, we find that
\[
H^i(X_\et,\Z/p^n\Z)\otimes_{\Z/p^n\Z} \OO_C/p^n\to H^i(X_\et,\OO_X^+/p^n)
\]
is an almost isomorphism. Passing to the inverse limit over $n$ (cf. \cite[Lemma 3.18]{ScholzeHodge}), we get that
\[
H^i(X_\proet,\hat{\Z}_p)\otimes_{\Z_p} \OO_C\to H^i(X_\proet,\hat{\OO}_X^+)
\]
is an almost isomorphism; also, $H^i(X_\proet,\hat{\Z}_p) = H^i_\et(X,\Z_p)$, with the usual definition of the right-hand side. Inverting $p$, we get an isomorphism
\[
H^i_\et(X,\Q_p)\otimes_{\Q_p} C\cong H^i(X_\proet,\hat{\OO}_X)\ .
\]
It remains to compute the right-hand side. For this, we use the projection $\nu: X_\proet\to X_\et$, and the spectral sequence
\[
E_2^{ij} = H^i(X_\et,R^j\nu_\ast \hat{\OO}_X)\Rightarrow H^{i+j}(X_\proet,\hat{\OO}_X) = H^{i+j}_\et(X,\Q_p)\otimes_{\Q_p} C\ ;
\]
this reduces us to the next proposition.
\end{proof}

\begin{prop} Let $X/C$ be a smooth adic space. Let $\nu: X_\proet\to X_\et$ be the projection. Then $\OO_{X_\et}\buildrel\cong\over\longrightarrow \nu_\ast \hat{\OO}_X$, and there is a natural isomorphism $\Omega_{X_\et}^1(-1)\buildrel\cong\over\longrightarrow R^1\nu_\ast \hat{\OO}_X$; it induces isomorphisms (via taking the exterior power, and cup products)
\[
\Omega_{X_\et}^j(-j)\buildrel\cong\over\longrightarrow R^j\nu_\ast \hat{\OO}_X
\]
for all $j\geq 0$.
\end{prop}

\begin{proof} First, we prove that $\mathcal{E} = R^1\nu_\ast \hat{\OO}_X$ is a locally free $\OO_{X_\et}$-module of rank $\dim X$, such that
\[
\bigwedge^j \mathcal{E}\cong R^j\nu_\ast \hat{\OO}_X
\]
for all $j\geq 0$. This can be checked locally, so we can assume that there is an \'etale map $X\to \mathbb{T}^n$ that factors as a composite of rational embeddings and finite \'etale maps, where
\[
\mathbb{T}^n = \Spa(C\langle T_1^{\pm 1},\ldots,T_n^{\pm 1}\rangle,\OO_C\langle T_1^{\pm 1},\ldots,T_n^{\pm 1}\rangle)
\]
is the $n$-dimensional torus. We have the natural pro-finite \'etale cover
\[
\tilde{\mathbb{T}}^n = \Spa(C\langle T_1^{\pm 1/p^\infty},\ldots,T_n^{\pm 1/p^\infty},\OO_C\langle T_1^{\pm 1/p^\infty},\ldots,T_n^{1/p^\infty}\rangle)\ .
\]
By pullback, it induces a pro-finite \'etale cover $\tilde{X}\to X$, and one has
\[
H^i(X_\proet,\hat{\OO}_X) = H^i_\cont(\Z_p^n,\OO_{\tilde{X}}(\tilde{X}))\ .
\]
Thus, the statement follows from \cite[Lemma 4.5, Lemma 5.5]{ScholzeHodge}: First, Lemma 4.5 shows that
\[
\OO_{\tilde{X}}(\tilde{X}) = \OO_X(X)\hat{\otimes}_{C\langle T_1^{\pm 1},\ldots,T_n^{\pm 1}\rangle} C\langle T_1^{\pm 1/p^\infty},\ldots,T_n^{\pm 1/p^\infty}\rangle\ ,
\]
and then Lemma 5.5 computes this group.

In particular, we find that $\OO_{X_\et}\buildrel\cong\over\longrightarrow \nu_\ast \hat{\OO}_X$. It remains to prove that $\mathcal{E}\cong \Omega_{X_\et}^1(-1)$. For this, we prove the following lemma.
\end{proof}

\begin{lem} Consider the exact sequence
\[
0\to \hat{\Z}_p(1)\to \varprojlim_{\times p} \OO_X^\times\to \OO_X^\times\to 0
\]
on $X_\proet$, where $\hat{\Z}_p(1) = \varprojlim \mu_{p^n}$. It induces a boundary map
\[
\OO_{X_\et}^\times = \nu_\ast \OO_X^\times\to R^1\nu_\ast \hat{\Z}_p(1)\ .
\]
There is a unique $\OO_{X_\et}$-linear map $\Omega_{X_\et}^1\to \mathcal{E} = R^1\nu_\ast \hat{\OO}_X(1)$ such that the diagram
\[\xymatrix{
\OO_{X_\et}^\times\ar[d]^{d\log} \ar[r] & R^1\nu_\ast \hat{\Z}_p(1)\ar[d]\\
\Omega_{X_\et}^1\ar[r] & R^1\nu_\ast \hat{\OO}_X(1)
}\]
commutes. This map is an isomorphism.
\end{lem}

\begin{proof} The assertion is local, so we can again assume that $X\to \mathbb{T}^n$ is a composition of finite \'etale maps and rational subsets; in particular, $X=\Spa(R,R^\circ)$ is affinoid. Let $T_1,\ldots,T_n$ be the coordinates on $\mathbb{T}^n$; then $d\log(T_i)\in \Omega_{X_\et}^1$ are an $\OO_{X_\et}$-basis. Their images are prescribed uniquely by the requirement of the lemma, thus the map is unique if it exists. Moreover, the local computation shows that this map is an isomorphism.

It remains to prove existence. For this, we have to check that the diagram
\[\xymatrix{
R^\times = H^0(X_\proet,\OO_X^\times)\ar[d]^{d\log} \ar[r] & H^1(X_\proet,\hat{\Z}_p(1))\ar[d]\\
\Omega_{R/C}^1 = H^0(X_\proet,\Omega_X^1)\ar[r] & H^1(X_\proet,\hat{\OO}_X(1))
}\]
commutes, where the lower map is defined to be the one pinned down by the $d\log(T_i)$. First, by the local computation, $H^1(X_\proet,\hat{\OO}_X(1))$ is a finite free $R$-module. In particular, it carries a natural topology. We claim that via both maps, the image of $R^{\circ\times}$ in $H^1(X_\proet,\hat{\OO}_X(1))$ is bounded (the commutativity of the diagram implies a posteriori the fact that the image of $R^\times$ is bounded). For the map over the upper right corner, this follows by observing that it factors over $H^1(X_\proet,\hat{\OO}_X^+(1))$. For the map over the lower left corner, observe that on $R^{\circ\times}$, $d\log$ factors over $\Omega_{R^\circ/\OO_C}^1$, where the latter is an $R^\circ$-module of finite type (as $R^\circ$ is topologically finitely generated by \cite{BoschGuentzerRemmert}).

Now write $C = \varinjlim A_i$ as the filtered direct limit of $\Q_p$-algebras $A_i$ of finite type. By resolution of singularities, we may assume that all $A_i$ are smooth. We may thus find formally smooth $\Q_p$-algebras $B_i$ topologically of finite type with $C=\widehat{\varinjlim B_i}$. We may also assume that there are \'etale maps $Y_i=\Spa(B_i,B_i^\circ)\to \mathbb{T}^{n_i}$ for some $n_i$. For $i$ large enough, $X$ comes as base extension from $X_i\to Y_i$ (use that finite \'etale covers and rational subsets on an inverse limits of adic spaces come from a finite level). Let $X_i = \Spa(R_i,R_i^\circ)$, where now $R_i$ is topologically of finite type over $\Q_p$. Moreover, given any $f\in R^\times$ and $m\geq 1$, one can find an $i$ and some $f_i\in R_i^\times$ such that $\frac{f}{f_i}\in 1 + p^m R^\circ$. By the boundedness statement on the image of $R^{\circ\times}$, and because elements of $1+p^m R^\circ$ are $p^{m-2}$-th powers of elements of $1+p^2 R^\circ\subset R^{\circ\times}$ (say $m\geq 2$), it suffices to prove the statement for $f=f_i$.

Now we recall Faltings's extension from \cite[Corollary 6.14]{ScholzeHodge}; this is an exact sequence
\[
0\to \hat{\OO}_{X_i}(1)\to \mathcal{F}_i\to \hat{\OO}_{X_i}\otimes_{\OO_{X_i}} \Omega_{X_i}^1\to 0
\]
of $\hat{\OO}_{X_i}$-modules on $(X_i)_\proet$. Moreover, there is an exact commutative diagram
\[\xymatrix{
0\ar[r] & \hat{\Z}_p(1)\ar[r]\ar[d] & \varprojlim_{\times p} \OO_{X_i}^\times\ar[d]\ar[r] &\OO_{X_i}^\times\ar[d]^{d\log}\ar[r] & 0\\
0\ar[r] & \hat{\OO}_{X_i}(1)\ar[r] & \mathcal{F}_i\ar[r] & \hat{\OO}_{X_i}\otimes_{\OO_{X_i}} \Omega_{X_i}^1\ar[r] & 0
}\]
on $(X_i)_\proet$. Here, an element of $\varprojlim_{\times p} \OO_{X_i}^\times$ gives an element $U\in \hat{\OO}_{X_i}^{\flat\times}$, and an element $V\in \OO_{X_i}^\times$; then $\log(V/[U])\in \OO\B_{\dR,X_i}^+$ lies in the kernel of $\Theta$, and defines an element of $\mathcal{F}_i$; this gives the middle vertical map. Commutativity is a direct check; note that $d[U] = 0$. Applying $g_i^\ast$ (where $g_i: X_\proet\to (X_i)_\proet$) to the diagram and tensoring the lower sequence with $\hat{\OO}_X$ over $g_i^\ast \hat{\OO}_{X_i}$ gives an exact commutative diagram
\[\xymatrix{
0\ar[r] & g_i^\ast \hat{\Z}_p(1)\ar[r]\ar[d] & g_i^\ast \varprojlim_{\times p} \OO_{X_i}^\times\ar[d]\ar[r] & g_i^\ast \OO_{X_i}^\times\ar[d]\ar[r] & 0\\
0\ar[r] & \hat{\OO}_X(1)\ar[r] & \mathcal{F}_i^\prime\ar[r] & \hat{\OO}_X\otimes_{\OO_X} (\Omega_X^1)^\prime\ar[r] & 0
}\]
on $X_\proet$, where $\mathcal{F}_i^\prime$ is some sheaf of $\hat{\OO}_X$-modules, and $(\Omega_X^1)^\prime = g_i^\ast \Omega_{X_i}^1\otimes_{g_i^\ast \OO_{X_i}} \OO_X$ is a free $\OO_X$-module with a projection $(\Omega_X^1)^\prime\to \Omega_X^1$. The kernel comes from $\Omega_{Y_i}^1$, and is generated by the image of $\OO_{Y_i}^\times$ via $d\log$ (using that $Y_i$ is \'etale over $\mathbb{T}^{n_i}$). Now look at the associated commutative diagram of boundary maps:
\[\xymatrix{
H^0(X_\proet,g_i^\ast \OO_{X_i}^\times)\ar[r]\ar[d] & H^1(X_\proet,g_i^\ast \hat{\Z}_p(1))\ar[d]\\
H^0(X_\proet,\hat{\OO}_X\otimes_{\OO_X} (\Omega_X^1)^\prime)\ar[r] & H^1(X_\proet,\hat{\OO}_X(1))
}\]
Note also that
\[\xymatrix{
H^0(X_\proet,g_i^\ast \OO_{X_i}^\times)\ar[r]\ar[d] & H^1(X_\proet,g_i^\ast \hat{\Z}_p(1))\ar[d]\\
R^\times\ar[r] & H^1(X_\proet,\hat{\Z}_p(1))
}\]
commutes, as one can relate the two exact sequences defining the boundary maps. Using the two diagrams together, one finds that the map
\[
H^0(X_\proet,\hat{\OO}_X\otimes_{\OO_X} (\Omega_X^1)^\prime)\to H^1(X_\proet,\hat{\OO}_X(1))
\]
defined by the $\hat{\OO}_X$-linear extension $\mathcal{F}_i^\prime$ is $\hat{\OO}_X$-linear, and maps $d\log(T_i)$ to the correct elements. Moreover, it factors over
\[
\Omega_{R/C}^1 = H^0(X_\proet,\hat{\OO}_X\otimes_{\OO_X} \Omega_X^1)\ :
\]
For this, we have to check that it kills all $d\log(h)$, where $h$ is one of the coordinates of $\mathbb{T}^{n_i}$, pulled back to $Y_i$. But then $h$ gives an element of $C$, and it has in $C$ a sequence of $p$-power roots; it follows that $h$ lifts to $H^0(X_\proet,g_i^\ast\varprojlim_{\times p} \OO_{X_i}^\times)$, so that it vanishes under the boundary map.

It follows that
\[
H^0(X_\proet,\hat{\OO}_X\otimes_{\OO_X} (\Omega_X^1)^\prime)\to H^1(X_\proet,\hat{\OO}_X(1))
\]
is the composite of the projection
\[
H^0(X_\proet,\hat{\OO}_X\otimes_{\OO_X} (\Omega_X^1)^\prime)\to \Omega_{R/C}^1 = H^0(X_\proet,\hat{\OO}_X\otimes_{\OO_X} \Omega_X^1)
\]
with the map
\[
\Omega_{R/C}^1\to H^1(X_\proet,\hat{\OO}_X(1))
\]
pinned down by the $d\log(T_i)$. Now, use the commutativity of the diagrams for $f_i\in H^0(X_\proet,g_i^\ast \OO_{X_i}^\times)$ to conclude.
\end{proof}

\newpage

\section{A $p$-adic analogue of Riemann's classification of complex abelian varieties}

In this section, we explain a result on $p$-divisible groups over $\OO_C$, where $C$ is an algebraically closed complete extension of $\Q_p$, proved in joint work with J. Weinstein, \cite{ScholzeWeinstein}. Although in \cite{ScholzeWeinstein}, perfectoid spaces are used at several points, the material of this section is independent of the theory of perfectoid spaces, except for Subsection \ref{HodgeTateSeqPDiv}, which is however not needed to state and prove the main result discussed below. Thus, the reader may prefer to skip Subsection \ref{HodgeTateSeqPDiv}, except for reading the statement of Fargues's theorem \ref{FarguesHodgeTate}.

\subsection{Riemann's theorem}
Let us first recall the classical theory over the complex numbers $\mathbb{C}$.

\begin{definition} A complex torus is a connected compact complex Lie group $T$.
\end{definition}

\begin{lem} A complex torus $T$ is commutative.
\end{lem}

\begin{proof} Let $\mathcal{O}_{T,0}$ be the local ring at $0$, with maximal ideal $\mathfrak{m}$. The adjoint action
\[
T\to \GL(\mathcal{O}_{T,0}/\mathfrak{m}^n)
\]
is trivial, as $T$ is compact and connected, and the right-hand side is affine.
\end{proof}

Let $\mathfrak{t}$ be the Lie algebra of $T$; it is a finite-dimensional $\mathbb{C}$-vector space, $\mathfrak{t}\cong \mathbb{C}^g$. The exponential map
\[
\exp: \mathfrak{t}\to T
\]
makes $\mathfrak{t}$ the universal covering of $T$, and if we let $\Lambda=\ker(\exp)\subset \mathfrak{t}$, then $T=\mathfrak{t}/\Lambda$. Thus, $\Lambda\subset \mathfrak{t}$ is a discrete cocompact subgroup, i.e. a lattice $\Lambda\cong \Z^{2g}$. Thus, we arrive at the classification of complex tori.

\begin{prop} The category of complex tori is equivalent to the category of pairs $(\mathfrak{t},\Lambda)$, where $\mathfrak{t}$ is a finite-dimensional $\mathbb{C}$-vector space, and $\Lambda\subset \mathfrak{t}$ is a lattice.
\end{prop}

One can reformulate this classification in terms of Hodge structures.

\begin{definition}\begin{altenumerate}
\item[{\rm (i)}] A $\Z$-Hodge structure of weight $-1$ is a finite free $\Z$-module $\Lambda$ together with a $\mathbb{C}$-subvectorspace $V\subset \Lambda\otimes_{\Z} \mathbb{C}$, such that $V\oplus \bar{V}\buildrel\cong\over\longrightarrow \Lambda\otimes_{\Z} \mathbb{C}$.
\item[{\rm (ii)}] A polarization on a $\Z$-Hodge structure $(\Lambda,V)$ of weight $-1$ is an alternating form
\[
\psi: \Lambda\otimes \Lambda\to 2\pi i\Z
\]
such that $\psi(x,Cy)$ is a symmetric positive definite form on $\Lambda\otimes_{\Z} \mathbb{R}$, where $C$ is Weil's operator on $\Lambda\otimes_{\Z} \mathbb{C}\cong V\oplus \bar{V}$, acting as $i$ on $V$, and as $-i$ on $\bar{V}$.
\end{altenumerate}
\end{definition}

We note that it follows that complex tori are equivalent to $\Z$-Hodge structures of weight $-1$, via mapping $(\mathfrak{t},\Lambda)$ to $(\Lambda,V)$, with $V=\ker(\Lambda\otimes \mathbb{C}\to \mathfrak{t})$.

\begin{definition} A complex abelian variety is a projective complex torus.
\end{definition}

\begin{thm}[Riemann] The category of complex abelian varieties is equivalent to the category of polarizable $\Z$-Hodge structures of weight $-1$.
\end{thm}

Let us call this theorem, stating an abstract equivalence between some geometric objects (abelian varieties) with some Hodge-theoretic data, the 'Hodge-theoretic perspective'. In this case over $\mathbb{C}$, we have seen that this equivalence has a very direct geometric meaning, which we call 'the geometric perspective': All complex tori of dimension $g$ have the same universal cover $\mathbb{C}^g$, and thus are of the form $\mathbb{C}^g/\Lambda$ for a lattice $\Lambda\subset \mathbb{C}^g$. When can one form the quotient $\mathbb{C}^g/\Lambda$? Always as a complex manifold, sometimes (as determined by Riemann) as an algebraic variety.

\subsection{The Hodge-Tate sequence for abelian varieties and $p$-divisible groups}\label{HodgeTateSeqPDiv}

Let $C$ be an algebraically closed complete extension of $\Q_p$. Let $X/C$ be a proper smooth scheme. We recall the following results.

\begin{thm} There is a Hodge-de Rham spectral sequence
\[
E_1^{ij} = H^j(X,\Omega_X^i)\Rightarrow H^{i+j}_\dR(X)\ .
\]
It degenerates at $E_1$. One gets a decreasing Hodge-de Rham filtration $\Fil^\bullet H^i_\dR(X)$, with
\[
\Fil^q H^i_\dR(X) / \Fil^{q+1} H^i_\dR(X) = H^{i-q}(X,\Omega_X^q)\ .
\]
\end{thm}

\begin{thm} There is a Hodge-Tate spectral sequence
\[
E_2^{ij} = H^i(X,\Omega_X^j)(-j)\Rightarrow H^{i+j}_\et(X,\Q_p)\otimes_{\Q_p} C\ .
\]
It degenerates at $E_2$. One gets a decreasing Hodge-Tate filtration $\Fil^\bullet (H^i_\et(X,\Q_p)\otimes_{\Q_p} C)$, with
\[
\Fil^q (H^i_\et(X,\Q_p)\otimes_{\Q_p} C) / \Fil^{q+1} (H^i_\et(X,\Q_p)\otimes_{\Q_p} C) = H^q(X,\Omega_X^{i-q})(q-i)\ .
\]
\end{thm}

\begin{example} Let $A/C$ be an abelian variety, with universal vector extension $EA\to A$, and $p$-adic Tate module $\Lambda$. Recall that $\Lie EA$ is dual to $H^1_\dR(A)$, and $\Lambda$ is dual to $H^1_\et(A,\Z_p)$. One has two short exact sequences (where $A^\ast$ is the dual abelian variety)
\[
0\to (\Lie A^\ast)^\ast\to \Lie EA\to \Lie A\to 0\ ,
\]
\[
0\to (\Lie A)(1)\to \Lambda\otimes_{\Z_p} C\to (\Lie A^\ast)^\ast\to 0\ .
\]
\end{example}

One has the following compatibility of the Hodge-Tate sequence with duality.

\begin{prop}\label{DualityAbVar} Let $A/C$ be an abelian variety, and let $A^\ast/C$ be the dual abelian variety. Identify $H^0(A,\Omega_A^1) = (\Lie A)^\ast$, $H^1(A,\OO_A) = \Lie A^\ast$, and similarly for $A^\ast$. The two sequences
\[
0\to \Lie A^\ast\to H^1_\et(A,\Z_p)\otimes_{\Z_p} C\to (\Lie A)^\ast(-1)\to 0\ ,
\]
\[
0\to \Lie A\to H^1_\et(A^\ast,\Z_p)\otimes_{\Z_p} C\to (\Lie A^\ast)^\ast(-1)\to 0
\]
are dual to each other under the Weil pairing
\[
H^1_\et(A,\Z_p)\otimes_{\Z_p} H^1_\et(A^\ast,\Z_p)\to \Z_p(-1)\ .
\]
\end{prop}

\begin{proof} Let $\mathfrak{P}$ be the Poincar\'e bundle on $A\times A^\ast$. Unraveling definitions, the proposition follows from the following compatibility for the first Chern class $c_1(\mathfrak{P})$ on $A\times A^\ast$: $c_1^\dR(\mathfrak{P})$ encodes the duality between $H^1(A,\OO_A)$ and $H^0(A^\ast,\Omega_{A^\ast}^1)$, whereas $c_1^\et(\mathfrak{P})$ encodes the Weil pairing.
\end{proof}

\begin{lem} Let $X/C$ be a proper smooth scheme, and $\mathcal{L}$ a line bundle on $X$; regard $\mathcal{L}$ as an element of $H^1(X,\OO_X^\times)$. Define
\[
c_1^\dR(\mathcal{L})\in H^1(X,\Omega_X^1)
\]
as the image of $\mathcal{L}\in H^1(X,\OO_X^\times)$ under $d\log: \OO_X^\times\to \Omega_X^1$. Define
\[
c_1^\et(\mathcal{L})\in H^2(X_\proet,\hat{\Z}_p(1)) = H^2_\et(X,\Z_p)(1)
\]
as the image of $\mathcal{L}\in H^1(X,\OO_X^\times) = H^1(X_\proet,\OO_X^\times)$ under the boundary map associated to the short exact sequence
\[
0\to \hat{\Z}_p(1)\to \varprojlim_{\times p} \OO_X^\times\to \OO_X^\times\to 0
\]
on $X_\proet$. Then $c_1^\et(\mathcal{L})\otimes 1\in H^2_\et(X,\Z_p)\otimes_{\Z_p} C(1)$ lies in $\Fil^1(H^2_\et(X,\Z_p)\otimes_{\Z_p} C)(1)$, and maps to
\[
c_1^\dR(\mathcal{L})\in H^1(X,\Omega_X^1) = \gr^1(H^2_\et(X,\Z_p)\otimes_{\Z_p} C)(1)\ .
\]
\end{lem}

\begin{rem} The statement is true more generally for $X/C$ a proper smooth rigid-analytic variety for which the Hodge-Tate spectral sequence degenerates.
\end{rem}

\begin{proof} One has a map $\OO_X^\times\to \hat{\Z}_p(1)[1]\to \hat{\OO}_X(1)[1]$ in the derived category of sheaves on $X_\proet$. The associated map of Leray spectral sequences shows that $c_1^\et(\mathcal{L})$ lies in
\[
\Fil^1(H^2_\et(X,\Z_p)\otimes_{\Z_p} C)(1) = \Fil^1(H^2(X_\proet,\hat{\OO}_X(1)))\ ;
\]
note that $R\nu_\ast \OO_X^\times = \OO_{X_\et}^\times$ by \cite[Corollary 3.17 (i)]{ScholzeHodge}. The final statement follows from the commutative diagram
\[\xymatrix{
\OO_{X_\et}^\times\ar[d]^{d\log} \ar[r] & R^1\nu_\ast \hat{\Z}_p(1)\ar[d]\\
\Omega_{X_\et}^1\ar[r]^\cong & R^1\nu_\ast \hat{\OO}_X(1)
}\]
upon applying $H^1(X_\et,-)$.
\end{proof}

Now assume that $A$ has good reduction, i.e. we have an abelian variety $A/\OO_C$, where $\OO_C\subset C$ is the ring of integers. Then we can describe everything in terms of the $p$-divisible group $G=A[p^\infty]$. Indeed, we have the universal vector extension $EG\to G$, and the $p$-adic Tate module $\Lambda$ of $G$. One has the short exact sequence of finite free $\OO_C$-modules
\[
0\to (\Lie G^\ast)^\ast\to \Lie EG\to \Lie G\to 0\ ,
\]
where $G^\ast$ denotes the Serre dual $p$-divisible group.

\begin{thm}[Fargues, {\cite[Ch.2,App.C]{FarguesTwoTowers}}]\label{FarguesHodgeTate} There is a complex of finite free $\OO_C$-modules
\[
0\to (\Lie G)(1)\buildrel{\alpha_{G^\ast}^\ast(1)}\over\longrightarrow \Lambda\otimes_{\Z_p} \OO_C\buildrel{\alpha_G}\over\longrightarrow (\Lie G^\ast)^\ast\to 0\ .
\]
Its cohomology groups are killed by $p^{1/(p-1)}$.
\end{thm}

\begin{rem} Fargues uses the following direct definition of $\alpha_G$. Take any $\lambda\in \Lambda$. Thus,
\[
\lambda\in \varprojlim_n G[p^n](C) = \varprojlim_n G[p^n](\OO_C) = \Hom_{\OO_C}(\Q_p/\Z_p,G)\ .
\]
Thus, by duality, one gets a map $G^\ast\to \mu_{p^\infty}$, which on Lie algebras gives map $\Lie G^\ast\to \Lie \mu_{p^\infty}\cong \OO_C$; i.e. we get an element of $(\Lie G^\ast)^\ast$.
\end{rem}

\begin{prop} Let $A/\OO_C$ be an abelian variety, $G=A[p^\infty]$, with Tate module $\Lambda$. The two Hodge-Tate sequences (the first from Theorem \ref{HodgeTateSpecSeq}, the second from Fargues's theorem)
\[
0\to \Lie A^\ast\otimes_{\OO_C} C\to H^1_\et(A_C,\Z_p)\otimes_{\Z_p} C\to (\Lie A)^\ast\otimes_{\OO_C} C(-1)\to 0\ ,
\]
\[
0\to \Lie G\otimes_{\OO_C} C(1)\to \Lambda\otimes_{\Z_p} C\to (\Lie G^\ast)^\ast\otimes_{\OO_C} C\to 0
\]
are dual to each other.
\end{prop}

\begin{proof} By definition, the second exact sequence is compatible with duality $G\mapsto G^\ast$. By Proposition \ref{DualityAbVar}, the first exact sequence is compatible with duality $A\mapsto A^\ast$; as $A^\ast[p^\infty] = G^\ast$, we are reduced to checking that
\[
H^1_\et(A_C,\Z_p)\otimes_{\Z_p} C(1)\to (\Lie A)^\ast\otimes_{\OO_C} C
\]
agrees with
\[
\Lambda^\ast\otimes_{\Z_p} C(1)\to (\Lie G)^\ast\otimes_{\OO_C} C\ .
\]
Recall that the first map is defined as the composite
\[\begin{aligned}
H^1_\et(A_C,\Z_p)\otimes_{\Z_p} C(1)&\cong H^1(A_{C,\proet},\hat{\OO}_A)(1)\\
&\to H^0(A_{C,\et},R^1\nu_\ast \hat{\OO}_A)(1)\cong H^0(A_{C,\et},\Omega_A^1) = (\Lie A)^\ast\otimes_{\OO_C} C\ .
\end{aligned}\]
Consider $G$ as a formal scheme over $\Spf \OO_C$, and let $G_\eta$ be its generic fibre. Then $G_\eta\subset A_C$ is a rigid-analytic open subset. Fix an element $\lambda\in \Lambda^\ast(1)$; it corresponds to a map $\Q_p/\Z_p\to G^\ast$ of $p$-divisible groups over $\OO_C$, thus to a morphism $G\to G^\prime = \mu_{p^\infty}$ over $\OO_C$. One generic fibres, it induces a morphism $G_\eta\to G^\prime_\eta$. Note that the latter is just an open unit ball. We get a commutative diagram
\[\xymatrix{
H^1(A_{C,\proet},\hat{\Z}_p)(1)\ar[r]\ar[d] & H^1(A_{C,\proet},\hat{\OO}_A)(1)\ar[r]\ar[d]& H^0(A_{C,\et},R^1\nu_\ast \hat{\OO}_A)(1)\ar[r]^\cong\ar[d]& H^0(A_{C,\et},\Omega_A^1)\ar@{^(->}[d]\\
H^1(G_{\eta,\proet},\hat{\Z}_p)(1)\ar[r] & H^1(G_{\eta,\proet},\hat{\OO}_{G_\eta})(1)\ar[r]& H^0(G_{\eta,\et},R^1\nu_\ast \hat{\OO}_{G_\eta})(1)\ar[r]^\cong& H^0(G_{\eta,\et},\Omega_{G_\eta}^1)\\
& H^1(G^\prime_{\eta,\proet},\hat{\OO}_{G^\prime_\eta})(1)\ar[u]\ar[r]& H^0(G^\prime_{\eta,\et},R^1\nu_\ast \hat{\OO}_{G^\prime_\eta})(1)\ar[u]\ar[r]^\cong & H^0(G^\prime_{\eta,\et},\Omega_{G^\prime_\eta}^1)\ar[u]\\
& H^1(G^\prime_{\eta,\proet},\hat{\Z}_p)(1)\ar[uul]\ar[u] & H^0(G^\prime_{\eta,\proet},\OO_{G^\prime_\eta}^\times)\ar[l]\ar[r]^{=} & H^0(G^\prime_\eta,\OO_{G^\prime_\eta}^\times)\ar[u]^{d\log}
}\]
Identify $G^\prime = \mu_{p^\infty} = \Spf \Z_p[[T]]$ in the usual way; now look at what happens to the element $1+T$ as an element of the lower-right group. Under $d\log$, it maps to the standard basis element of $\Lie G^\prime\otimes_{\OO_C} C$, thus the image in
\[
H^0(G_{\eta,\et},\Omega_{G_\eta}^1) = (\Lie G)^\ast\otimes_{\OO_C} H^0(G_{\eta,\et},\OO_{G_{\eta,\et}})
\]
is given by $\alpha_{G^\ast}(\lambda)\otimes 1$.

On the other hand, mapping $1+T$ around the lower-left hand corner, we use Kummer theory first. Note that the $H^1(-,\hat{\Z}_p)(1)$-groups on the left classify $\Z_p(1)$-covers. Extracting a sequence of $p$-power roots of $1+T$ amounts exactly to the tower corresponding to multiplication by $p$-powers on $G^\prime_\eta$. On the other hand, $\lambda\in \Lambda^\ast(1) = H^1(A_{C,\proet},\hat{\Z}_p)(1)$ gives a similar tower over $A_C$; the two towers become equal after restriction to $G_\eta$. Now the result follows from the commutativity of the diagram, and the injectivity of the vertical upper right map.
\end{proof}

\subsection{A $p$-adic analogue: The Hodge-theoretic perspective}

The preceding discussion gives a functor from $p$-divisible groups over $\OO_C$ to pairs $(\Lambda,W)$, where $\Lambda$ is a finite free $\Z_p$-module, and $W\subset \Lambda\otimes_{\Z_p} C$ is a $C$-subvectorspace.

\begin{thm}[{\cite[Theorem 5.2.1]{ScholzeWeinstein}}] This functor is an equivalence of categories.
\end{thm}

Let us make a series of remarks.

\begin{rem}\begin{altenumerate}
\item[{\rm (i)}] In an unpublished manuscript, Fargues (\cite{FarguesPDivGroups}) had previously proved fully faithfulness, and more.
\item[{\rm (ii)}] This is the first instance of a classification of $p$-divisible groups in terms of linear algebra (instead of $\sigma$-linear algebra, as usual in Dieudonn\'e theory).
\item[{\rm (iii)}] Let $\kappa$ be the residue field of $\OO_C$. Then, by reduction to the special fibre, one has a functor from $p$-divisible groups over $\OO_C$ to $p$-divisible groups over $\kappa$. Recall that the latter are classified by Dieudonn\'e modules $(M,F,V)$, where $M$ is a finite free $W(\kappa)$-module, $F: M\to M$ is a $\sigma$-linear map, and $V: M\to M$ is a $\sigma^{-1}$-linear map, such that $FV=VF=p$. Here, $\sigma: W(\kappa)\to W(\kappa)$ is the lift of Frobenius on $\kappa$.

Using these equivalences of categories, one gets a functor $(\Lambda,W)\mapsto (M,F,V)$, which the author does not know how to describe. Describing this functor amounts to an integral comparison between the \'etale and crystalline cohomology of $p$-divisible groups. However, in the current situation over $C$, there is no Galois action on the Tate module, which would usually be used in Fontaine's theory.\footnote{Using the Fargues-Fontaine curve, \cite{FarguesFontaineDurham}, one can describe the functor up to isogeny.}
\item[{\rm (iv)}] If $0\to G_1\to G_2\to G_3\to 0$ is an exact sequence of $p$-divisible groups over $\OO_C$, then the corresponding sequences $0\to W_1\to W_2\to W_3\to 0$ and $0\to \Lambda_1\to \Lambda_2\to \Lambda_3\to 0$ are exact. However, the converse is not true.
For example, take $G_2$ of height $2$ and dimension $1$, with supersingular special fibre. Then there exists a complex
\[
0\to \Q_p/\Z_p\to G_2\to \mu_{p^\infty}\to 0
\]
that becomes exact on $\Lambda$'s and $W$'s.
\item[{\rm (v)}] Given $(\Lambda,W)$, one gets a $p$-divisible group $G$ over $\OO_C$ by the theorem, and thus an $\OO_C$-lattice
\[
W^\circ = (\Lie G)(1)\subset W=(\Lie (G)(1))\otimes_{\OO_C} C\subset \Lambda\otimes_{\Z_p} C\ .
\]
Can one describe $W^\circ$ directly in terms of $(\Lambda,W)$? We note that there is a second natural $\OO_C$-lattice $(W^\circ)^\prime = (\Lambda\otimes \OO_C)\cap W$. Then Fargues's theorem implies
\[
p^{1/(p-1)} (W^\circ)^\prime\subset W^\circ\subset (W^\circ)^\prime\ ;
\]
however, none of the two inclusions is an equality in general.
\item[{\rm (vi)}] Coming back to (iv), we note that an exact sequence of $p$-divisible groups gives an exact sequence on $\Lambda$'s and $W^\circ$'s. Is the converse true? It seems not unreasonable to hope that the answer is yes.\footnote{V. Pilloni has communicated to us a simple proof of this statement, using Fargues's results on degrees for finite locally free group schemes, \cite{FarguesFiniteGroups}.}
\item[{\rm (vii)}] Let $K$ be a discretely valued complete extension of $\Q_p$ with perfect residue field. Then one can reprove the following theorem of Breuil, \cite{Breuil}, (in case $p\neq 2$), and Kisin, \cite{Kisin}, (in general) 'by descent':
\end{altenumerate}
\end{rem}

\begin{thm} The category of $p$-divisible groups over $\OO_K$ is equivalent to the category of lattices in crystalline representations of $\Gal(\bar{K}/K)$ with Hodge-Tate weights $0$, $1$.
\end{thm}

\subsection{A $p$-adic analogue: The geometric perspective}

\begin{definition} Let $R$ be a ring which is $p$-torsion, and let $G/R$ be a $p$-divisible group (considered as an fpqc sheaf on schemes over $R$). Define the universal cover of $G$ as $\tilde{G} = \varprojlim_{\times p} G$, as an fpqc sheaf on schemes over $R$. Moreover, let $TG=\ker(\tilde{G}\to G)$; one has a short exact sequence of fpqc sheaves
\[
0\to TG\to \tilde{G}\to G\to 0\ .
\]
\end{definition}

We note that $TG$ is a sheafified version of the Tate module: For any $R$-algebra $R^\prime$,
\[
TG(R^\prime) = \varprojlim G[p^n](R^\prime) = \Hom_{R^\prime}(\Q_p/\Z_p,G)\ .
\]
Also, $\tilde{G} = TG\otimes_{\Z_p} \Q_p$.

\begin{prop}\begin{altenumerate}
\item[{\rm (i)}] The functor $G\mapsto \tilde{G}$ turns isogenies into isomorphisms.
\item[{\rm (ii)}] Let $S\to R$ be a surjection with nilpotent kernel, and $G_S$ a lift of $G$ to $S$. Then for any $S$-algebra $S^\prime$,
\[
\tilde{G}_S(S^\prime)\buildrel\cong\over\longrightarrow \tilde{G}(S^\prime\otimes_S R)\ .
\]
In particular, $\tilde{G}_S$ depends only on $\tilde{G}$: One can consider $\tilde{G}$ as a crystal on the infinitesimal site of $R$.
\end{altenumerate}
\end{prop}

\begin{proof} Part (i) is clear. For part (ii) we may assume $S^\prime = S$. Recall that by a result of Illusie, \cite{IllusieDefBT}, the categories of $p$-divisible groups up to isogeny over $R$ and $S$ are equivalent. Thus,
\[
\tilde{G}_S(S) = \Hom_S(\Q_p/\Z_p,G_S)[p^{-1}] = \Hom_R(\Q_p/\Z_p,G)[p^{-1}] = \tilde{G}(R)\ .
\]
\end{proof}

In some cases, one can write down $\tilde{G}$: One gets examples of perfectoid spaces, by taking generic fibres!

\begin{prop}[{\cite[Corollary 3.1.5]{ScholzeWeinstein}}]\label{UnivCovPerf} Let $G$ be a connected $p$-divisible group over $W(\kappa)$, for some perfect field $\kappa$. Then there is an isomorphism of fpqc sheaves,
\[
\tilde{H}\cong \Spf W(\kappa)[[T_1^{1/p^\infty},\ldots,T_d^{1/p^\infty}]]\ .
\]
\end{prop}

Fix an embedding $\overline{\mathbb{F}}_p\hookrightarrow \OO_C/p$.

\begin{thm}[{\cite[Theorem 5.1.4 (i)]{ScholzeWeinstein}}] Let $G$ be a $p$-divisible group over $\OO_C$. Then there is a $p$-divisible group $H/\overline{\mathbb{F}}_p$ (unique up to isogeny) and a quasi-isogeny $\rho: G\otimes_{\OO_C} \OO_C/p\to H\otimes_{\overline{\mathbb{F}}_p} \OO_C/p$.
\end{thm}

In particular, one finds that $\tilde{G}\cong \tilde{H}_{\OO_C}$, where the latter denotes the evaluation of $\tilde{H}$ on $\OO_C$, considered as a crystal on the infinitesimal site. Thus, in any dimension, there are only finitely many possibilities for the universal cover of $G$, and these are given by the Dieudonn\'e-Manin classification of $p$-divisible groups up to isogeny.

Now fix a $p$-divisible group $H$ over $\overline{\mathbb{F}}_p$, of height $h$ and dimension $d$. Note that for any $p$-divisible group $G$ over $\OO_C$, we have the $\Z_p$-lattice $\Lambda\subset \tilde{G}(\OO_C)$, where $\Lambda = TG(\OO_C)$ denotes the Tate module. In particular, we get a fully faithful functor from the category of pairs $(G,\rho)$, where $G$ is a $p$-divisible group over $\OO_C$, and $\rho: G\otimes_{\OO_C} \OO_C/p\to H\otimes_{\overline{\mathbb{F}}_p} \OO_C/p$ is a quasi-isogeny, to the category of $\Z_p$-lattices $\Lambda\subset \tilde{H}(\OO_C)$. Here, we use $\rho$ to identify $\tilde{G}$ and $\tilde{H}$.

Thus, as in the case over $\mathbb{C}$, one can ask the question for which $\Z_p$-lattices $\Lambda\subset \tilde{H}(\OO_C)$ one can form the quotient $\tilde{H}/\Lambda$ to get a $p$-divisible group. In order to state the answer, we need the following proposition.

\begin{prop} For any $p$-divisible group $G$ over $\OO_C$, there is a natural logarithm map $\log_G: G(\OO_C)\to \Lie G\otimes C$. One gets a short exact sequence
\[
0\to \Lambda[p^{-1}]\to \tilde{G}(\OO_C)\to \Lie G\otimes C\to 0\ ,
\]
where $\Lambda=TG(\OO_C)$ is the Tate module of $G$. Moreover, there is a natural 'quasi-logarithm' map $\qlog: \tilde{H}(\OO_C)\to M(H)\otimes C$, where $M(H)$ is the (covariant) Dieudonn\'e module, such that for any $(G,\rho)$, the diagram
\[\xymatrix{
\tilde{H}(\OO_C)\ar[d]^\cong\ar[rr]^\qlog && M(H)\otimes C\ar[d]\\
\tilde{G}(\OO_C)\ar[r] & G(\OO_C)\ar[r]^\log & \Lie G\otimes C
}\]
commutes. Here, the map $M(H)\otimes C\to \Lie G\otimes C$ comes from Grothendieck-Messing theory.
\end{prop}

\begin{thm}[{\cite[Theorem D]{ScholzeWeinstein}}]\label{ClassLattices} The category of pairs $(G,\rho)$ is equivalent to the category of $\Z_p$-lattices $\Lambda\subset \tilde{H}(\OO_C)$ such that the cokernel $V=\coker(\Lambda\otimes C\to M(H)\otimes C)$ is of dimension $d$, and the sequence
\[
0\to \Lambda[p^{-1}]\to \tilde{H}(\OO_C)\to V\to 0
\]
is exact.
\end{thm}

\newpage

\section{Rapoport-Zink spaces}

Using these results, we show in \cite{ScholzeWeinstein} that Rapoport-Zink spaces at infinite level carry a natural structure as a perfectoid space. More precisely, consider a $p$-divisible group $H$ over $\bar{\mathbb{F}}_p$ of dimension $d$ and height $h$. Rapoport-Zink, \cite{RapoportZink}, define the following deformation space.

\begin{thm}[\cite{RapoportZink}] The functor sending a $W(\bar{\mathbb{F}}_p)$-algebra $R$ on which $p$ is nilpotent to the set of isomorphism classes of pairs $(G,\rho)$, where $G/R$ is a $p$-divisible group and
\[
\rho: H\otimes_{\bar{\mathbb{F}}_p} R/p\to G\otimes_R R/p
\]
is a quasi-isogeny, is representable by a formal scheme $\mathcal{M}$.
\end{thm}

Moreover, its generic fibre $\mathcal{M}_\eta$, considered as an adic space, has a natural system of coverings $\mathcal{M}_n\to \mathcal{M}_\eta$ parametrizing isomorphisms $(\Z/p^n\Z)^h\to G[p^n]$. Each of them is a rigid-analytic variety. However, the inverse limit of the $\mathcal{M}_n$ does not make sense within rigid geometry. In \cite{ScholzeWeinstein}, we prove however the following result. For simplicity, we work with the base-change $\mathcal{M}_{n,K}$ of $\mathcal{M}_n$ to $\Spa(K,\OO_K)$, where $K/\Q_p$ is a perfectoid field.

\begin{thm} There is a unique (up to unique isomorphism) perfectoid space $\mathcal{M}_{\infty,K}$ over $\Spa(K,\OO_K)$ such that $\mathcal{M}_{\infty,K}\sim \varprojlim \mathcal{M}_{n,K}$. In fact, $\mathcal{M}_{\infty,K}$ is a locally closed subspace of the generic fibre of $\tilde{H}_{\OO_K}^h$, defined by certain explicit conditions.
\end{thm}

Here, we use $\sim$ as in Definition \ref{InverseLimit}.

\begin{rem} We note that the generic fibre of $\tilde{H}_{\OO_K}^h$ is a perfectoid space, see Proposition \ref{UnivCovPerf} in the connected case. The map
\[
\mathcal{M}_{\infty,K}\to (\tilde{H}_{\OO_K})^h_\eta
\]
is precisely the map from Theorem \ref{ClassLattices} on $C$-valued points, sending a deformation to the corresponding lattice in the universal cover.
\end{rem}

We remark that the description of $\mathcal{M}_{\infty,K}$ as an explicit locally closed subspace of the generic fibre of $\tilde{H}_{\OO_K}^h$ allows us to prove duality isomorphisms between Rapoport-Zink spaces as isomorphisms of perfectoid spaces. We refer to \cite{ScholzeWeinstein} for details. The abstract setup is however as in the following proposition, which gives the expected cohomological consequences (which were not included in \cite{ScholzeWeinstein}).

\begin{prop} Let $C$ be an algebraically closed and complete extension of $\Q_p$. Let $G$, $\check{G}$ be two $p$-adic reductive groups, and let $\mathcal{M}_U$, $\check{\mathcal{M}}_{\check{U}}$ be two towers of partially proper smooth adic spaces over $C$ parametrized by compact open subgroups $U\subset G(\Q_p)$, resp. $\check{U}\subset \check{G}(\Q_p)$, and with an action of $G(\Q_p)$, resp. $\check{G}(\Q_p)$, on the tower, such that for $U^\prime\subset U$ an open normal subgroup, $\mathcal{M}_{U^\prime}$ is a finite \'etale Galois cover of $\mathcal{M}_U$ with Galois group $U/U^\prime$, respectively the similar statement for the tower $\check{\mathcal{M}}_{\check{U}}$. Moreover, assume $\check{G}(\Q_p)$ acts continuously on each $\mathcal{M}_U$, and $G(\Q_p)$ acts continuously on each $\check{\mathcal{M}}_{\check{U}}$, such that all group actions are compatible and commute.

Finally, assume that there is a perfectoid space $\mathcal{M}$ over $C$ with a continuous action of $G(\Q_p)\times \check{G}(\Q_p)$ such that
\[
\mathcal{M}\sim \varprojlim_U \mathcal{M}_U\ ,\ \mathcal{M}\sim \varprojlim_{\check{U}} \check{\mathcal{M}}_{\check{U}}\ ,
\]
where both maps to the inverse limit are $G(\Q_p)\times \check{G}(\Q_p)$-equivariant. Let $\ell\neq p$ be a prime.

\begin{altenumerate}
\item[{\rm (i)}] For any $m\geq 1$, there are $G(\Q_p)\times \check{G}(\Q_p)$-equivariant isomorphisms
\[
\varinjlim_U H^i_c(\mathcal{M}_U,\Z/\ell^m\Z)\cong H_c^i(\mathcal{M},\Z/\ell^m\Z)\cong \varinjlim_{\check{U}} H^i_c(\check{\mathcal{M}}_{\check{U}},\Z/\ell^m\Z)\ .
\]
\item[{\rm (ii)}] The action of $\check{G}(\Q_p)$ on $H_c^i(\mathcal{M}_U,\Z_\ell)$ is smooth for any $U\subset G(\Q_p)$; similarly, the action of $G(\Q_p)$ on $H_c^i(\check{\mathcal{M}}_{\check{U}},\Z_\ell)$ is smooth for any $\check{U}\subset \check{G}(\Q_p)$.
\item[{\rm (iii)}] There is a $G(\Q_p)\times \check{G}(\Q_p)$-equivariant isomorphism
\[
\varinjlim_U H^i_c(\mathcal{M}_U,\Z_\ell)\cong \varinjlim_{\check{U}} H^i_c(\check{\mathcal{M}}_{\check{U}},\Z_\ell)\ .
\]
Moreover, both identify with the $G(\Q_p)\times \check{G}(\Q_p)$-smooth vectors in $H_c^i(\mathcal{M},\Z_\ell)$.
\end{altenumerate}
\end{prop}

\begin{rem} Some remarks about the definitions of the various objects involved. We recall that for partially proper spaces, cohomology with compact support is defined as the derived functor of the functor $\Gamma_c$ of taking sections with quasicompact support. This leads to the following equivalent definition. Fix some $U$, and let $V_{0U}\subset V_{1U}\subset\ldots \mathcal{M}_U$ be a sequence of qcqs open subsets exhausting $\mathcal{M}_U$. Let $V^\prime_{0U}\subset V^\prime_{1U}\subset \ldots \mathcal{M}_U$ be a second such sequence, such that $V_{kU}\subset V^\prime_{kU}$ is a strict inclusion for all $k$, i.e. the $\overline{V_{kU}}\subset V^\prime_{kU}$. Let $j_{kU}: V_{kU}\to V^\prime_{kU}$ be the open inclusion. Then
\[
H_c^i(\mathcal{M}_U,\Z/\ell^m\Z) = \varinjlim_k H^i(V^\prime_{kU},j_{kU!} \Z/\ell^m\Z)\ .\footnote{The latter groups are also usually denoted $H^i_c(V_{kU},\Z/\ell^m\Z)$. In that case however, they are not the derived functors of global sections with compact support: $V_{kU}$ is not partially proper.}
\]
The same applies for $\mathcal{M}$, i.e. at infinite level.

For $\Z_\ell$-cohomology, one has to take inverse limits, and one has to take some care about the order. We define
\[
H_c^i(\mathcal{M}_U,\Z_\ell) = \varinjlim_k H^i(V^\prime_{kU},j_{kU!} \Z_\ell) = \varinjlim_k \varprojlim_m H^i(V^\prime_{kU},j_{kU!} \Z/\ell^m\Z)\ .
\]
We remark that a result of Huber in the book \cite{Huber} ensures that the groups
\[
H^i(V^\prime_{kU},j_{kU!} \Z/\ell^m\Z) = H^i_c(V_{kU},\Z/\ell^m\Z)
\]
are finite. We define $H_c^i(\mathcal{M},\Z_\ell)$ in the same way.
\end{rem}

\begin{proof}\begin{altenumerate}
\item[{\rm (i)}] There are obvious $G(\Q_p)\times \check{G}(\Q_p)$-equivariant maps, and we have to prove that they are isomorphisms; thus we may restrict to one tower, $\mathcal{M}_U$. 

Fix some $U$, and let $V_{0U}\subset V_{1U}\subset\ldots \mathcal{M}_U$ be a sequence of qcqs open subsets exhausting $\mathcal{M}_U$. Let $V^\prime_{0U}\subset V^\prime_{1U}\subset \ldots \mathcal{M}_U$ be a second such sequence, such that $V_{kU}\subset V^\prime_{kU}$ is a strict inclusion for all $k$, i.e. the $\overline{V_{kU}}\subset V^\prime_{kU}$. Let $j_{kU}: V_{kU}\to V^\prime_{kU}$ be the open inclusion. Then
\[
H_c^i(\mathcal{M}_U,\Z/\ell^m\Z) = \varinjlim_k H^i(V^\prime_{kU},j_{kU!} \Z/\ell^m\Z)\ .
\]
For $U^\prime\subset U$, let us denote by $V_{kU^\prime}$ etc. the corresponding objects induced by base-change, as well as $V_k\subset \mathcal{M}$. It is enough to prove that
\[
\varinjlim_{U^\prime\subset U} H^i(V^\prime_{kU^\prime}, j_{kU^\prime!} \Z/\ell^m\Z)\to H^i(V^\prime, j_{k!} \Z/\ell^m\Z)
\]
is an isomorphism. But this follows from \cite[Corollary 7.18]{ScholzePerfectoidSpaces1}.
\item[{\rm (ii)}] It suffices to check for $U$ small enough (using a Hochschild-Serre spectral sequence), so we may assume $U$ is pro-$p$. Now, the map
\[
H^i_c(\mathcal{M}_U,\Z/\ell^m\Z)\to \varinjlim_{U^\prime\subset U} H^i_c(\mathcal{M}_{U^\prime},\Z/\ell^m\Z) = H^i_c(\mathcal{M},\Z/\ell^m\Z)
\]
is injective. (The transition map to $U^\prime\subset U$ normal has an inverse, given by averaging over $U/U^\prime$.) On the right-hand side, the action of $\check{G}(\Q_p)$ is continuous by the comparison to the other tower in part (i), thus it is continuous on $H^i_c(\mathcal{M}_U,\Z/\ell^m\Z)$.

The result follows by noting that for actions of pro-$p$-groups on finitely generated $\ell$-adic modules, smoothness is equivalent to continuity. (By definition, one may write the cohomology as a direct limit of finitely generated $\Z_\ell$-modules.)
\item[{\rm (iii)}] It suffices to identify the $G(\Q_p)\times \check{G}(\Q_p)$-smooth vectors in $H_c^i(\mathcal{M},\Z_\ell)$ with
\[
\varinjlim_U H^i_c(\mathcal{M}_U,\Z_\ell)\ .
\]
By parts (i) and (ii), the direct limit injects into the $G(\Q_p)\times \check{G}(\Q_p)$-smooth vectors in $H_c^i(\mathcal{M},\Z_\ell)$. On the other hand, take a vector $v\in H_c^i(\mathcal{M},\Z_\ell)$ which is invariant under a compact open subgroup $U\subset G(\Q_p)$, without loss of generality pro-$p$. By averaging over $U$ (which is possible, as $U$ is pro-$p$), we see that $v$ comes from $H_c^i(\mathcal{M}_U,\Z_\ell)$, as desired.
\end{altenumerate}
\end{proof}

\begin{rem} We remark that the proof shows that in part (iii), the $G(\Q_p)$-smooth vectors are the same as the $\check{G}(\Q_p)$-vectors, which are then the $G(\Q_p)\times \check{G}(\Q_p)$-smooth vectors identified in part (iii).
\end{rem}

In the equal characteristic case, Weinstein, \cite{Weinstein}, has considered the Lubin-Tate case explicitly. In that case, the theory of Drinfeld level structures give natural integral models of $\mathcal{M}_n$ for all $n$, showing that in fact
\[
\mathcal{M}_n\cong (\Spf \bar{\mathbb{F}}_p[[X_{1,n},\ldots,X_{h,n}]])_\eta\ ,
\]
where the canonical coordinates $X_{1,n},\ldots,X_{h,n}$ come from the level-$n$-structure. This permits one to show by explicit computation that
\[
\mathcal{M}_\infty\cong (\Spf \bar{\mathbb{F}}_p[[X_1^{1/p^\infty},\ldots,X_h^{1/p^\infty}]])_\eta\ ,
\]
which obviously has the desired property of being perfectoid. It should be noted that these spaces live over the base field $\F_q[[t]]$, but it is rather hard to write down the element $t$ as an element of $\bar{\mathbb{F}}_p[[X_1^{1/p^\infty},\ldots,X_h^{1/p^\infty}]]$ (a formula appears at the very end of \cite{FarguesTwoTowers}).

The paper \cite{ScholzeWeinstein} also contains a result on Dieudonn\'e theory.

\begin{thm}[{\cite[Theorem 4.1.4]{ScholzeWeinstein}}] Let $R$ be a ring of characteristic $p$ which is the quotient of a perfect ring by a finitely generated ideal. Then the Dieudonn\'e module functor on $p$-divisible groups is fully faithful up to isogeny.
\end{thm}

Our result is slightly more precise than that, and in the case that $R$ is perfect, it recovers the fact that the Dieudonn\'e module functor is fully faithful, not just up to isogeny. Interestingly, the proof of this fully faithfulness result requires the use of perfectoid spaces, and most notably the almost purity theorem! In fact, this result from Dieudonn\'e theory bridges the gap between universal covers of $p$-divisible groups, and Fontaine's rings. Indeed, one gets the following corollary, where $B_\cris^+$ is Fontaine's crystalline period ring (associated to $C$).

\begin{cor} Let $C$ be an algebraically closed complete extension of $\Q_p$. Let $H/\bar{\mathbb{F}}_p$ be a $p$-divisible group, and write $\tilde{H}$ for its universal cover (lifted canonically to $\OO_C$). Then
\[
\tilde{H}(\OO_C) = (M(H)\otimes B_\cris^+)^{\varphi = p}\ ,
\]
where $M(H)$ is the Lie algebra of the universal vector extension of $H$. Under this identification, the quasi-logarithm map
\[
\qlog: \tilde{H}(\OO_C)\to M(H)\otimes C
\]
gets identified with
\[
1\otimes \Theta: M(H)\otimes B_\cris^+\to M(H)\otimes C\ ,
\]
where $\Theta: B_\cris^+ \to C$ is Fontaine's map.
\end{cor}

This translates also Theorem \ref{ClassLattices} into $p$-adic Hodge theory terms.

\newpage

\section{Shimura varieties, and completed cohomology}

A result very similar to the result for Rapoport-Zink spaces holds for Shimura varieties. Let $\Sh_K$, $K\subset G(\mathbb{A}_f)$ be a Shimura variety of Hodge type associated to some reductive group $G$ over $\Q$, defined over the reflex field $E$. For convenience, let us assume that $\Sh_K$ is projective, so that we do not have to worry about compactifications. Let $\C_p$ be the completion of an algebraic closure of $E_{\mathfrak{p}}$, where $\mathfrak{p}|p$ is a chosen place of $E$, and denote by $\Sh_{K,\C_p}$ the adic space over $\Spa(\C_p,\OO_{\C_p})$ associated to the base-change of $\Sh_K$ to $\C_p$. The following result is work in progress.

\begin{thm}[{\cite{ScholzeTorsion}}] For any sufficiently small level $K^p\subset G(\mathbb{A}_f^p)$ away from $p$, there exists a perfectoid space $\Sh_{K^p,\C_p}$ over $\Spa(\C_p,\OO_{\C_p})$ such that $\Sh_{K^p,\C_p}\sim \varprojlim_{K_p} \Sh_{K_pK^p,\C_p}$.
\end{thm}

Let us explain a consequence of this theorem, when combined with the first result on $p$-adic Hodge theory. Recall Emerton's definition of $p$-adically completed cohomology groups, in the torsion case (where no completion has to be taken):
\[
H^i(K^p,\F_p) = \varinjlim_{K_p} H^i_\et(\Sh_{K_pK^p,\bar{\Q}},\F_p)\ .
\]

\begin{cor} For $i>\dim \Sh_K$, we have $H^i(K^p,\F_p)=0$.
\end{cor}

\begin{proof} First, we can rewrite
\[
H^i(K^p,\F_p) = \varinjlim_{K_p} H^i_\et(\Sh_{K_pK^p,\bar{\Q}},\F_p) = \varinjlim_{K_p} H^i_\et(\Sh_{K_pK^p,\C_p},\F_p)\ .
\]
It is enough to prove that $H^i(K^p,\F_p)\otimes_{\F_p} \OO_{\C_p}/p$ is almost zero. But now
\[
H^i(K^p,\F_p)\otimes_{\F_p} \OO_{\C_p}/p = \varinjlim_{K_p} (H^i_\et(\Sh_{K_pK^p,\C_p},\F_p)\otimes_{\F_p} \OO_{\C_p}/p)\ ,
\]
and the latter is almost equal to
\[
\varinjlim_{K_p} H^i_\et(\Sh_{K_pK^p,\C_p},\OO_{\Sh_{K_pK^p,\C_p}}^+/p)
\]
by Theorem \ref{Thm2}. Now we use that $\Sh_{K^p,\C_p}\sim \varprojlim_{K_p} \Sh_{K_pK^p,\C_p}$ (which implies in particular a similar relation among \'etale topoi), giving
\[
\varinjlim_{K_p} H^i_\et(\Sh_{K_pK^p,\C_p},\OO_{\Sh_{K_pK^p,\C_p}}^+/p) = H^i_\et(\Sh_{K^p,\C_p},\OO_{\Sh_{K^p,\C_p}}^+/p)\ .
\]
But note that for affinoid perfectoid spaces $X$, $H^i_\et(X,\OO_X^+/p)$ is almost zero for $i>0$. It follows that $H^i_\et(\Sh_{K^p,\C_p},\OO_{\Sh_{K^p,\C_p}}^+/p)$ and $H^i_\an(\Sh_{K^p,\C_p},\OO_{\Sh_{K^p,\C_p}}^+/p)$ are almost equal. But now standard bounds on the cohomological dimension of topological spaces give the desired vanishing result.
\end{proof}

The corollary implies Conjecture 1.5 of Calegari and Emerton, \cite{CalegariEmerton}, in the case of (compact) Shimura varieties of Hodge type, except for nonstrict instead of strict inequalities on the codimensions. More applications of these ideas will appear in \cite{ScholzeTorsion}.

\bibliographystyle{abbrv}
\bibliography{CDM}

\end{document}